\documentclass[oneside]{article}
\usepackage{amsmath}
\usepackage{amssymb}
\usepackage{bbm}
\usepackage[margin=3cm]{geometry}
\usepackage{hyperref}
\usepackage{tikz}
\usepackage{braket}
\usepackage{lmodern}
\usepackage{iftex}
\usepackage{listings}
\usepackage[]{graphicx}
\usepackage{subcaption}
\usepackage{pst-plot}
\newtheorem{thm}{Theorem}[section]

\newtheorem{theorem}[thm]{Theorem}

\newtheorem{corollary}[thm]{Corollary}

\newenvironment{proof}{{\bf Proof:}}{\hfill$\square$\vskip.5cm}

\newcommand{\R}{\mathbb{R}}

\newcommand{\C}{\mathbb{C}}

\title{Online minimum search for Brownian motion and the Cauchy process: Multiple approaches}
\author{Erik Wu\footnote{\href{mailto:erikwu35758@gmail.com}{erikwu35758@gmail.com}, }\hspace{3pt} 
 and
Shannon Starr\footnote{\href{mailto:slstarr@uab.edu}{slstarr@uab.edu}}
\\[3pt]
\small
{\normalsize ${}^{*}$} 
\large James Clemens High School\\
\small 11306 County Line Rd,
Madison, AL 35756\\[3pt]
{\normalsize ${}^{\dagger}$} 
\large Department of Mathematics, University of Alabama at Birmingham\\
\small 1402 10th Avenue South,
Birmingham, AL 35294--1241
}

\date{December 29, 2023}

\begin{document}

\maketitle

\begin{abstract}
 \setcounter{section}{0}
The distribution for the minimum of Brownian motion or the Cauchy process is well-known using the reflection principle. Here we consider the problem of finding the sample-by-sample minimum,
which we call the online minimum search. We consider the possibility of the golden search method, but we show quantitatively that the 
bisection method is more efficient.
In the bisection method there is a
hierarchical parameter, which tunes the depth to which
each sub-search is conducted, somewhat similarly to how a depth-first search works to generate a topological ordering on nodes.
Finally, we consider
the possibility of using harmonic measure, which
is a novel idea that has so far been unexplored.
\end{abstract}

\section{Introduction}

For multiple applications in modern computing algorithms, a key step is optimizing functions with many local extrema.
A canonical example is the problem of neural networks for computing algorithms 
such as recognizing hand-written letters \cite{ZakChong1,MilgramCherietSabourin}.
An older application is maximum likelihood estimation \cite{CasellaBerger}.
Usually, such functions have many parameters and the difficulty is in optimizing in a high-dimensional vector space.
But the 1-dimensional problem of finding the minimum of Brownian motion on an interval is similar in that
it possesses many local minima. Almost surely, it possesses one global minimum (since B.m.~minus its running minimum
is a reflected Brownian motion: see for example \cite{MortersPeres}).

We demonstrate the standard optimization approaches on this special problem.
The goal is to find a good approximation to the minimum with as few queries of the random Brownian motion function
as possible, in order to minimize the run-time of the algorithm.
Using this standard, we find that the golden search method is inferior to the bisection method. 
Here, the bisection is applied utilizing Monte Carlo and randomized methods.

In this bisection method, we decompose the interval into blocks and perform a hierarchical search in different blocks, selecting
random intervals to estimate the global minimum over each iteration.
This method is highly efficient, consistently. It is always accurate for sufficiently large iterations and hierarchical levels.
We describe this in Section \ref{sec:2}.

To demonstrate the greater accuracy and efficiency of the bisection method, we run large simulations to
model each method's average error and run times depending on specific combinations of inputted parameters.
We present this data in Section \ref{sec:3}.


The bisection method also applies to the Cauchy process.
Because it is a L\'evy process, the reflection principle gives the formula for the distribution of the minimum
of the Cauchy process on the interval \cite{Sato}.
The Cauchy process could be seen as an interesting alternative to the Brownian motion process
in finance, serving a pedagogical r\^ole.
It also has applications in other areas of science, which we describe in Section \ref{sec:5}.

Section \ref{sec:4} of our article is devoted to harmonic measure. If we are trying to find the minimum value, we could imagine
a Brownian particle diffusing up from infinitely far below our random function. Then the harmonic measure
gives the probability for it to hit each of the intervals defined so far.
The harmonic measure may be calculated using complex-analytic tools, especially the Schwarz-Christoffel
mapping, because we use the piece-wise linear approximation
to Brownian motion, for a path that has been sampled at a finite collection of points
in the algorithm up to the $n$th step.

\section{Potential Search Methods} \label{sec:2}

A textbook optimization technique is the Golden Section search (GSS) method. 
For a continuous unimodal function it allows an efficient calculation of the minimum point
\cite{ZakChong2}.
But many optimization methods are applied to functions with multiple local minima.
One may still apply the Golden Section method, and restart if one attains a criterion for having found a local minimum.
However, this is not guaranteed to be efficient.

\subsection{Golden Section Search} \label{subsection1}
Suppose that we want to find the unique minimum for a convex function $f:[0,1] \rightarrow \mathbb{R}$ without using the derivatives of $f$. Then we use
the golden section method. It is called the golden section because of how we choose points $0=t_1<t_2<t_3=1$, that we use to take values of $f$: namely
$f(0), f(t_1), f(t_2),$ and $f(t_3)$, since we take $t_1$ and $t_2$ to be $1/\varphi^2$ and $1/\varphi$ (where $\varphi$ is the golden ratio), respectively. The reason for this choice will be explained shortly.

Since $t_1+t_2 = 1$, the four points are symmetrically situated on the number line as shown.
\begin{figure}[h]
    \centering
    \includegraphics[width=7cm]{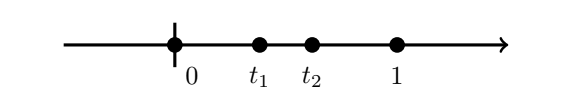}
    \caption{A symmetric number line representing the first iteration of a GSS}
\end{figure}

Now, for simplicity, we assume that $f(0), f(t_1), f(t_2),$and $f(1)$ are distinct points. Then we will either have $f(t_1) < f(t_2)$ or $f(t_1) > f(t_2)$.
If $f(t_1) < f(t_2)$, then we choose a new subinterval $[a',b'] = [0,t_2]$ to replace the original interval $[0,1]$. Similarly, If $f(t_1) > f(t_2)$, then we choose a new subinterval $[a',b'] = [t_1,1]$ to replace the original interval $[0,1]$. Two possible schematic representations for the graph of the function,
consistent with these two scenarios are shown:
\begin{figure}[h]
    \centering
    \includegraphics[width=6cm]{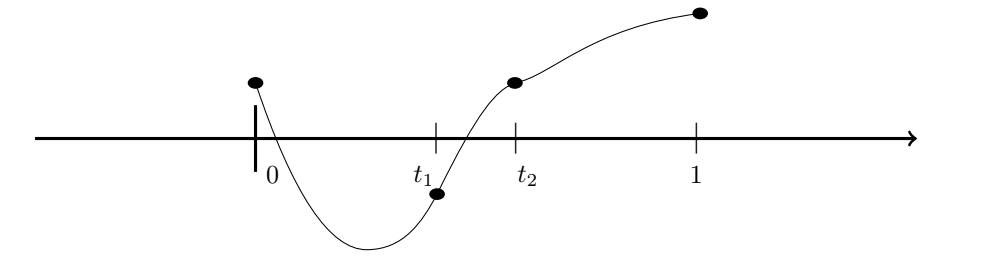}
    \caption{A possible graph when $f(t_2) > f(t_1)$}
\end{figure}

\begin{figure}[ht]
    \centering
    \includegraphics[width=6cm]{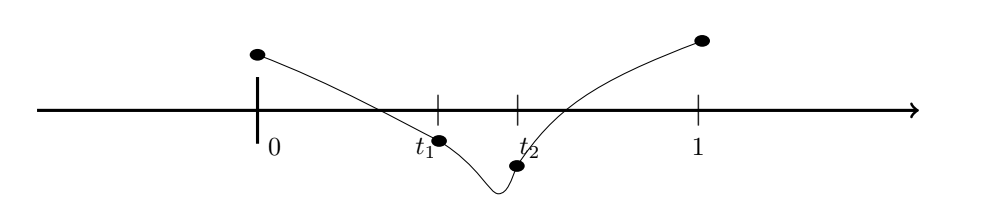}
    \caption{A possible graph when $f(t_1) > f(t_2)$}
\end{figure}

If $f(t_1) < f(t_2)$, then min({$f(s) : 0 \leq s \leq t_2$}) will not be $f(t_2)$. So there is a local minimum on the interval $[0, t_2]$ that at a point other
than $t_2$. Since $f$ is unimodal, this means that local minimum of $f$ on $[0, t_2]$ is actually the local minimum of $f$ on $[0,1]$. That is why we take 
$[a',b'] = [0,t_2]$ in this case. The same reasoning applies for the $f(t_1) > f(t_2)$ case, since if we find a local minimizer on $[t_1,1]$ other than at $t_1$, then it will be the local minimizer on $[0,1]$. That is why we choose $[a',b'] = [t_1, 1]$.

For the case where $[a,b] = [0,1]$, the reason we took $t_1 = 1/\varphi^2$ and $t_2 = 1/\varphi$ is as follows. If we take $[a',b'] = [0,t_2]$, then $t_0^{(1)} = 0$ and $t_3^{(1)} = 1/\varphi$. But then $t_2^{(1)} = t_3^{(1)}/\varphi = 1/\varphi^2 = t_1^{(0)}$. Thus, we save ourselves the computational effort of re-evaluating
$f(t_1^{(0)})$ when we need to calculate $f(t_2^{(1)})$. It is similar if $[a', b'] = [t_1, 1]$, since we have $t_1^{(1)}=t_1+(1-t_1)/\varphi^2=t_2$ (by the fact that $\varphi^2+1 = \varphi)$.

We can utilize the original golden-ratio oriented symmetry to find that in general, $t_1^{(k)} = t_0^{(k)} + (t_3^{(k)} - t_0^{(k)})/\varphi^2$ and $t_2^{(k)} = t_0^{(k)} + (t_3^{(k)} - t_0^{(k)})/\varphi$. Using this and the strategic storage of $f(t_0^{(k)}),f(t_1^{(k)}),f(t_2^{(k)}),$ and $f(t_3^{(k)})$, we can create a computational algorithm to find the minimum of a unimodal function.

A sample Golden Section Search algorithm call's associated graph (on the function $f(x) = \cos(x)$) using matplotlib is shown below.

\begin{figure}[h]
    \centering
    \includegraphics[width=5cm]{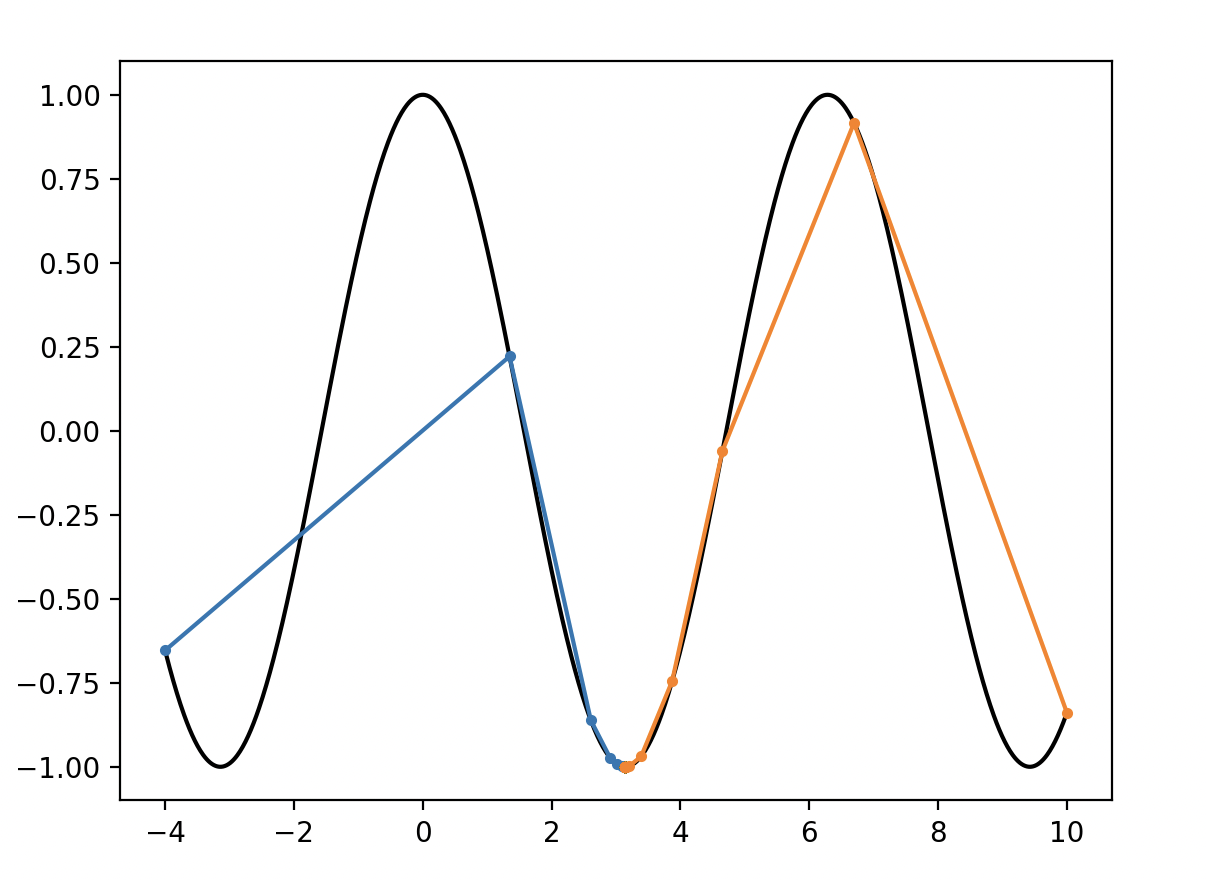}
    \caption{A sample GSS algorithm call on $f(x) = \cos(x)$}
\end{figure}

\subsection{Iterative Golden Section Search}
As seen above, the algorithm works well for unimodal functions; however, problems arise when the inputted function has multiple local minima. A possible example of this scenario is shown below for a Golden Section Search on the function $f(x) = x + \sin(x) + x\sin(x)$.

\begin{figure}[h]
    \centering
    \includegraphics[width=5cm]{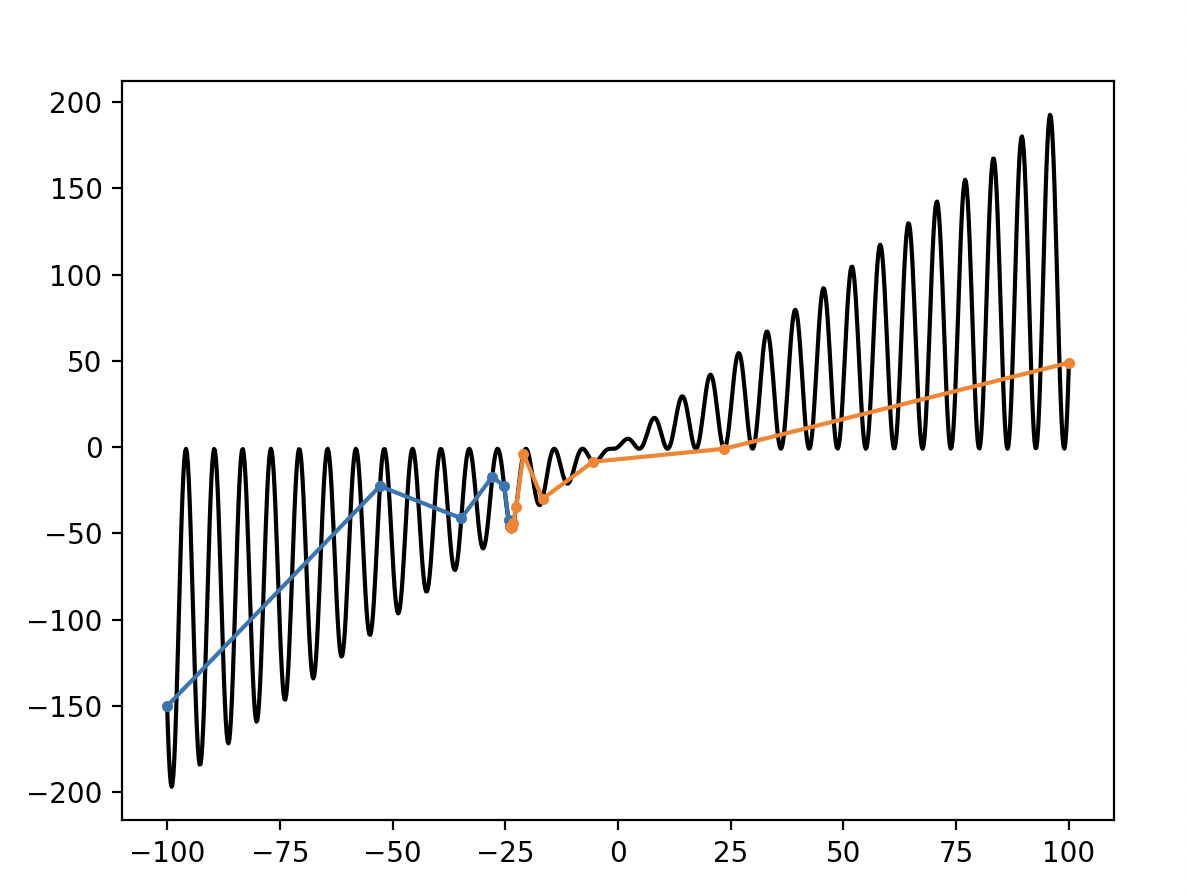}
    \caption{A sample GSS algorithm call on $f(x) = x + \sin(x) + x\sin(x)$}
\end{figure}
Applying the Golden Section Search to a Brownian Bridge, a clearly non-unimodal function, we see that the function does not find the global minimum, 
at least consistently. 
\begin{figure}[h]
    \centering
    \includegraphics[width=5cm]{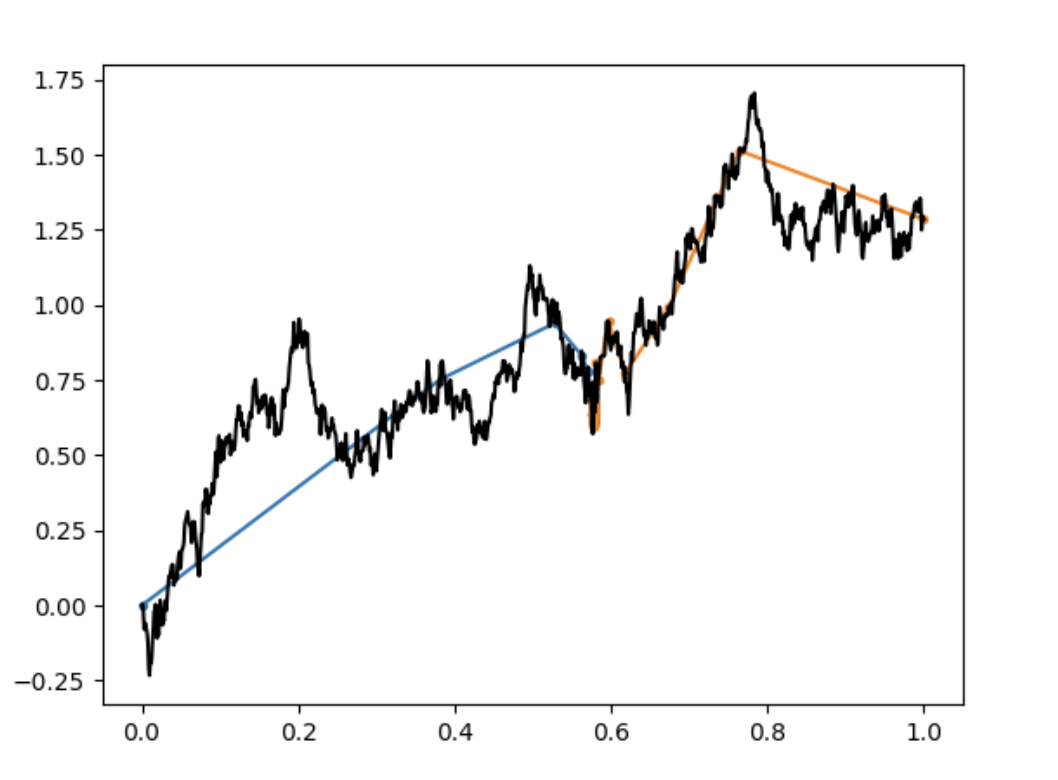}
    \caption{An example when GSS fails to find the global minimum on a Brownian Bridge}
\end{figure}

To fix this limitation on the Golden Section Search's requirement on uni-modality, we apply a robust approach. We first simulate a Brownian Bridge on $[0,1]$ 
with the end points being $(0,0)$ and $(1, W)$, where W is a standard, normal random variable (normal with mean 0 and variance 1). To fill in the points, we first find an intermediate point (specifically at $s_0 = 1/2$) utilizing the fact that $B(s) = sB(1) + \sqrt{s(1-s)} \cdot Z$ for some simulated standard, normal random variable $Z$ and $0 < s < 1$. Furthermore, for $0<r<s<1$, since $B(s) - B(r)$ and $B(1) - B(r)$ can be simulated like a new Brownian motion $\Tilde{B}(s-r)$ and $\Tilde{B}(1-r)$, respectively, we can let $B(s)$ equal
\begin{equation}
B(r) + \frac{s-r}{1-r}\bigl(B(1)-B(r)\bigr) + \biggl(\frac{(s-r)(1-s)}{1-r}\biggr)^{1/2} \cdot U
\end{equation}
for some standard normal random variable $U$ independent of what has been simulated thus far and $B(r)$ has already been found. Using this formula, we construct an algorithm to simulate new points on the Brownian Bridge (e.g., at $s_1 = 1/4$ or $s_2 = 3/4$ and so on by bisecting existing intervals) such that the final Bridge contains $2^n + 1$ simulated points where $n$ is a positive integer parameter. 

Now, with the simulated Brownian bridge, we partition the bridge (which is simulated on $[0,1]$) into $2^m$, where $m\in\mathbb{N}$, smaller intervals. On each of these intervals, we apply an iteration of the previously inaccurate Golden Section Search and add each found minimum into a specified "minimum-list". Once each sub-search has been completed, we first add the end points of the bridge as there is a possibility of the path going almost strictly upward or downward. Afterwards, we perform a simple search of the ``minimum-list'' to find a more accurate approximation of the online (sample-by-sample) minimum on a Brownian bridge.

Additionally, this algorithm takes in one more parameter: $\varepsilon$. This $\varepsilon$ parameter essentially specifies when the computer will stop each sub-search. More specifically, on the $k^{\text{th}}$ iteration of the sub-search, if $B(t_1^{(k)}) < B(t_2^{(k)})$, then the sub-search stops if $|B(t_3^{(k)}) - B(t_3^{(k+1)})| < \varepsilon$. Similarly, if $B(t_1^{(k)}) > B(t_2^{(k)})$, then the sub-search stops if $|B(t_0^{(k)}) - B(t_0^{(k+1)})| < \varepsilon$.

We call this method an Iterative Golden Section Search and figures are shown below to visually demonstrate how it works and how it compared to the original naive Golden Section Search approach in Subsection \ref{subsection1} when finding the minimum.
\pagebreak

\begin{figure}[h]
\centering
\begin{subfigure}{.5\textwidth}
  \centering
  \includegraphics[width=5cm]{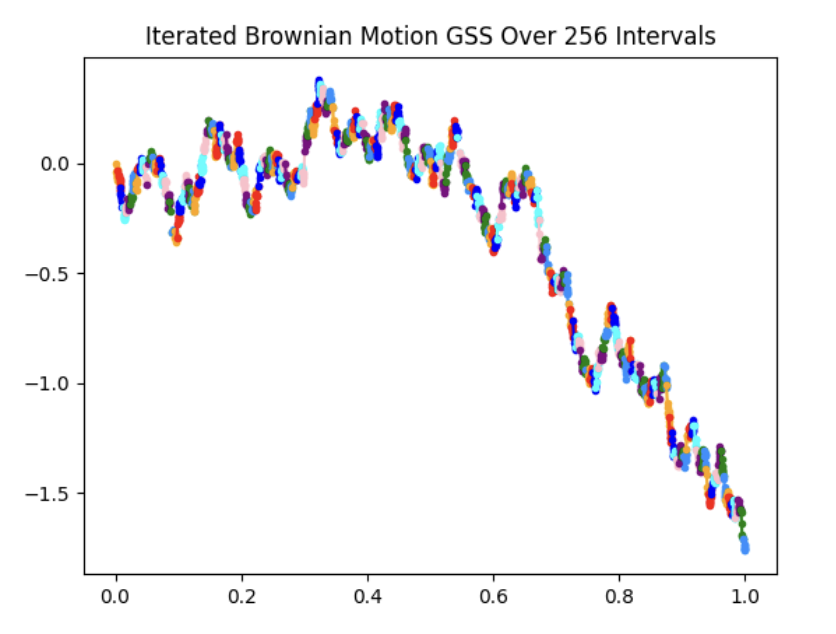}
  \caption{Visual graph of the iterative method}
  \label{fig:sub1}
\end{subfigure}%
\begin{subfigure}{.5\textwidth}
  \centering
  \includegraphics[width=5cm]{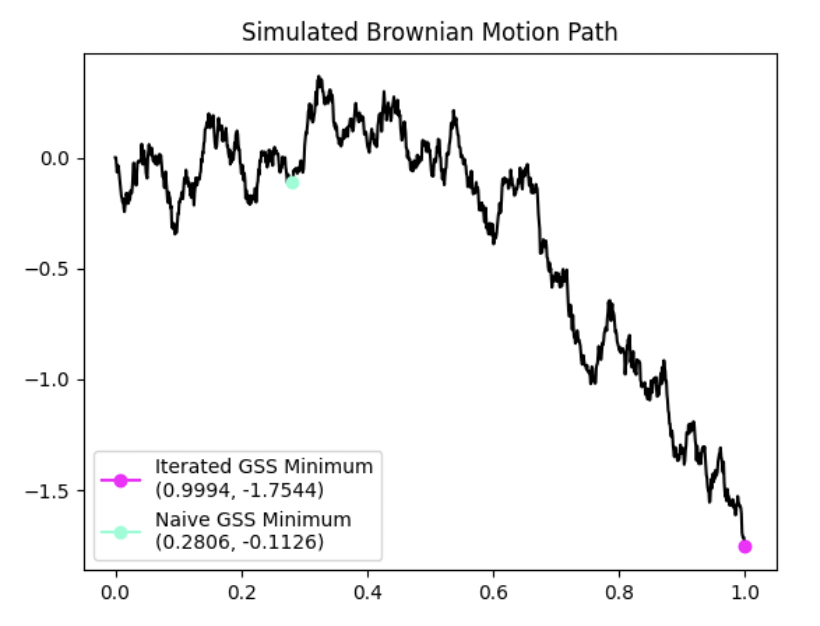}
  \caption{A comparison of the methods' minima estimates}
  \label{fig:sub2}
\end{subfigure}
\caption{The Iterative Golden Section Search}
\label{fig:test1}
\end{figure}

From Figure \ref{fig:test1} it is clear that the Iterative Golden Section Search did a better job in finding the minimum than the Naive Golden Section Search, but this program was simply shown on one simulated sample. Of course, there is a high likelihood of volatility between each algorithm's performance on any given sample; therefore, we must also analyze consistent accuracy that could hinder or benefit both of these aforementioned algorithms. 

\subsection{Analyzing the Iterative Golden Section Search and the Naive Approach} \label{subsection2}
To analyze the overall accuracy of both of these methods, we will do a simple comparison. For the Naive Approach, we will simply run the method for 
500 times, which is chosen in order to reduce volatility, and for each iteration, we find the absolute difference (we call this the error) between the estimated and actual minimum (this sequence of steps will repeated 8 times). In the end, we find the average error over the 500 iterations. Similarly, for the Iterative Golden Section Search, we will run 500 iterations over $2^n$ partitions for each $n$ where $n \in \mathbb{N}$ and $0<n<9$ (we define this interval for $n$ as higher values often take much longer running times for minimal difference in accuracy). In these first trials we will set $\varepsilon = 0.001$. The results of these comparisons are shown in the table below. 

\begin{table}[h]
\centering
\caption{Accuracy Comparison for the Iterative and Naive GSS Searches}
\begin{subfigure}{.5\textwidth}
  \centering
  \includegraphics[width=8cm]{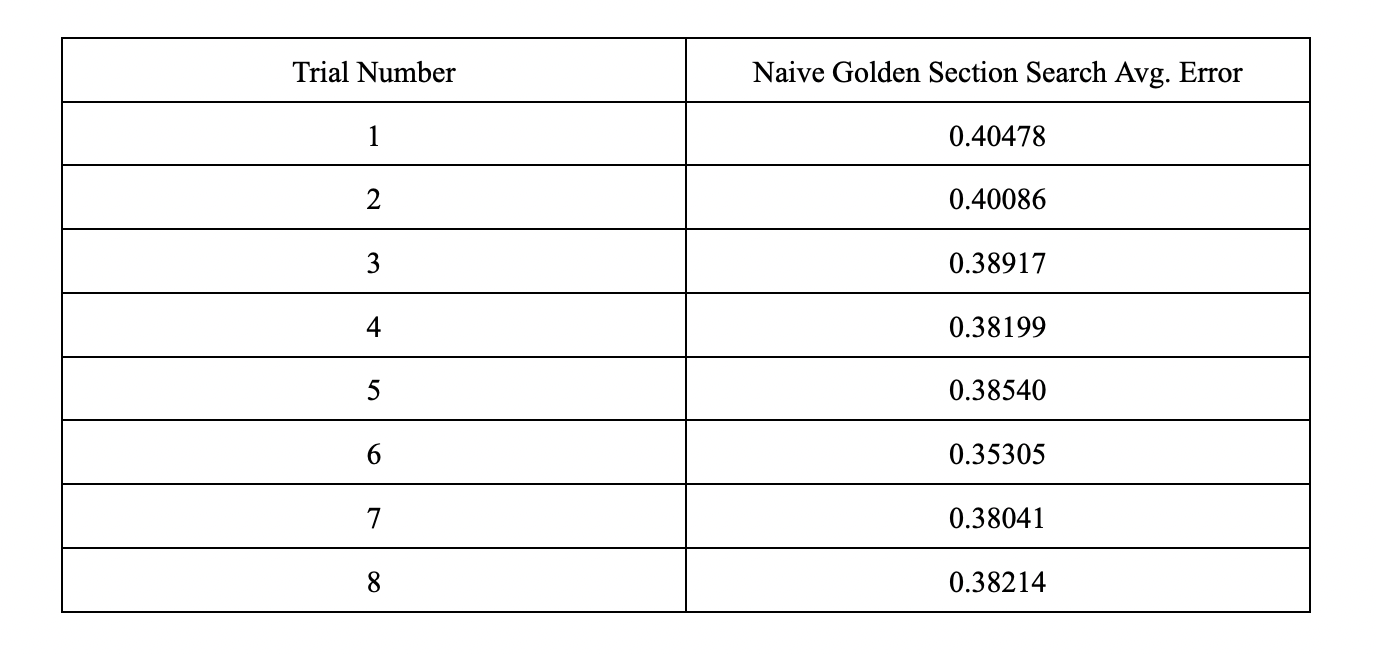}
  \label{fig:sub1}
\end{subfigure}%
\begin{subfigure}{.5\textwidth}
  \centering
  \includegraphics[width=8cm]{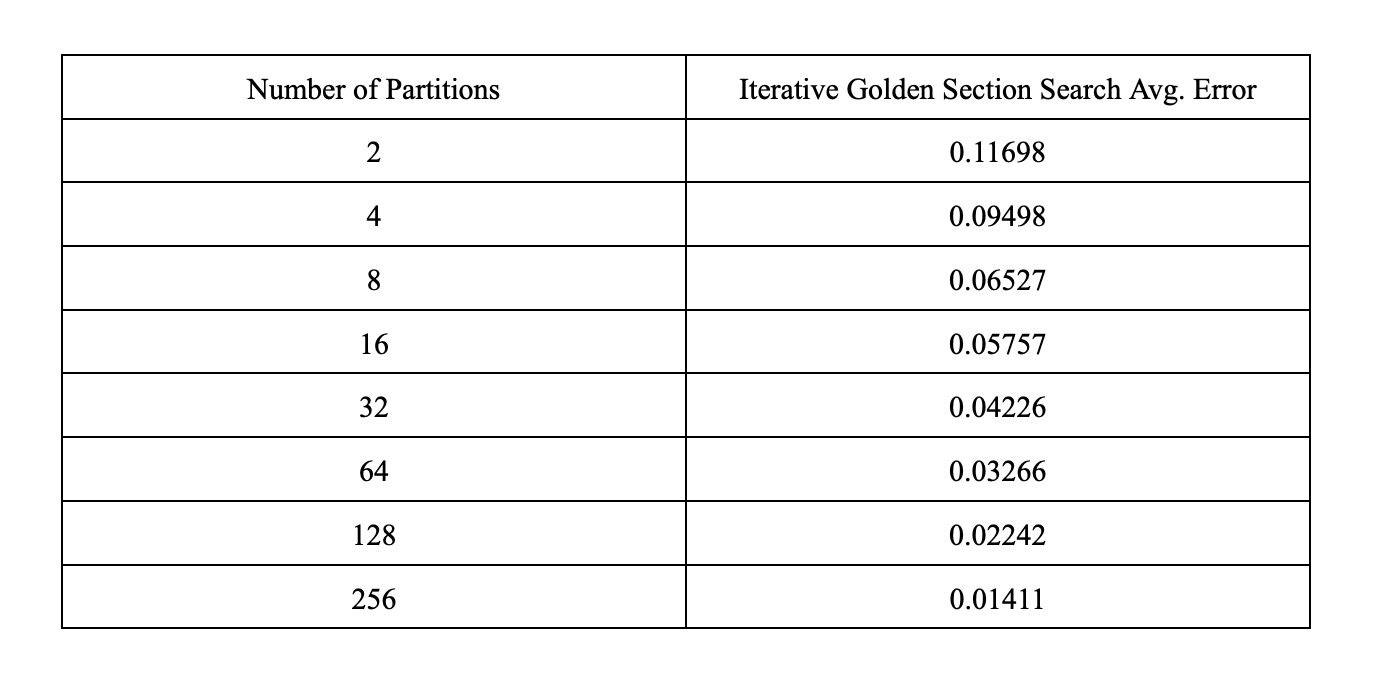}
  \label{fig:sub2}
\end{subfigure}
\label{table:test2}
\end{table}

From Table \ref{table:test2}, we see that the Iterative Golden Section Search is much better at finding the global minimum even at lower numbers of partitions. This data also supports the notion that greater number of  partitioned intervals allow for greater accuracy in determining the global minimum. In addition, the average error on the Naive Method demonstrates that the searching algorithm usually gets stuck in a local minimum, which is rarely a great estimation of the global minimum. However, there is a major issue with the Iterative Golden Section Search: scalability. Regarding scalability, the Iterative Golden Section Search does not scale well to other stochastic processes like the Cauchy process since it requires known values at golden-ratio related times, which are hard and inefficient to find without a well-defined, explicit bridge formula. 

We will now utilize the Bisection method (which is scalable to the Cauchy process) to find the global minimum of Brownian Motion due to its reliance on conveniently found values at integer times.

\subsection{Monte Carlo Bisection Method}
Generally, the Bisection method is used to find the root of a continuous function that is based upon the intermediate value theorem. In our case, we will adapt this method by incorporating random and Monte Carlo features to be applied to our continuous Brownian bridge path as follows. 

Similar to the algorithms related to the Golden Section Search, we will first simulate a Brownian bridge using the same methodology as before. Afterwards, we will utilize Monte-Carlo methods to find the global minimum of our Brownian bridge path. First, we will specify a parameter $l$ ($l \in \mathbb{N}$) that essentially tells the algorithm that we will simulate $2^l + 1$ Brownian bridge points. This algorithm will work by iteratively and randomly bisecting intervals. Starting with on $[0,1]$, we generate a random integer, either 0 or 1; if the integer is 0, then we focus now on $[0, 1/2]$, and if the integer is 1, then we focus now on $[1/2, 1]$. Now, we repeat this process continually for $r$ times, where r is a positive integer parameter and $r\leq l$. When the interval (we call it $[t_0, t_1]$) is bisected $r$ times, we store $(t_0 + t_1)/2$ in a ``minimum list'' as an estimation of the global minimum of our Brownian Bridge path; this process of randomly bisecting intervals and storing these values in a ``minimum list'' is repeated $g$ times, where $g$ is a positive integer parameter. Lastly, we simply find the minimum of the ``minimum list'', which is outputted as our approximation of the global minimum over the simulated Brownian bridge. A possible schematic representation of one iteration of this method is shown below. 

\begin{figure}[h]
    \centering
    \includegraphics[width=5cm]{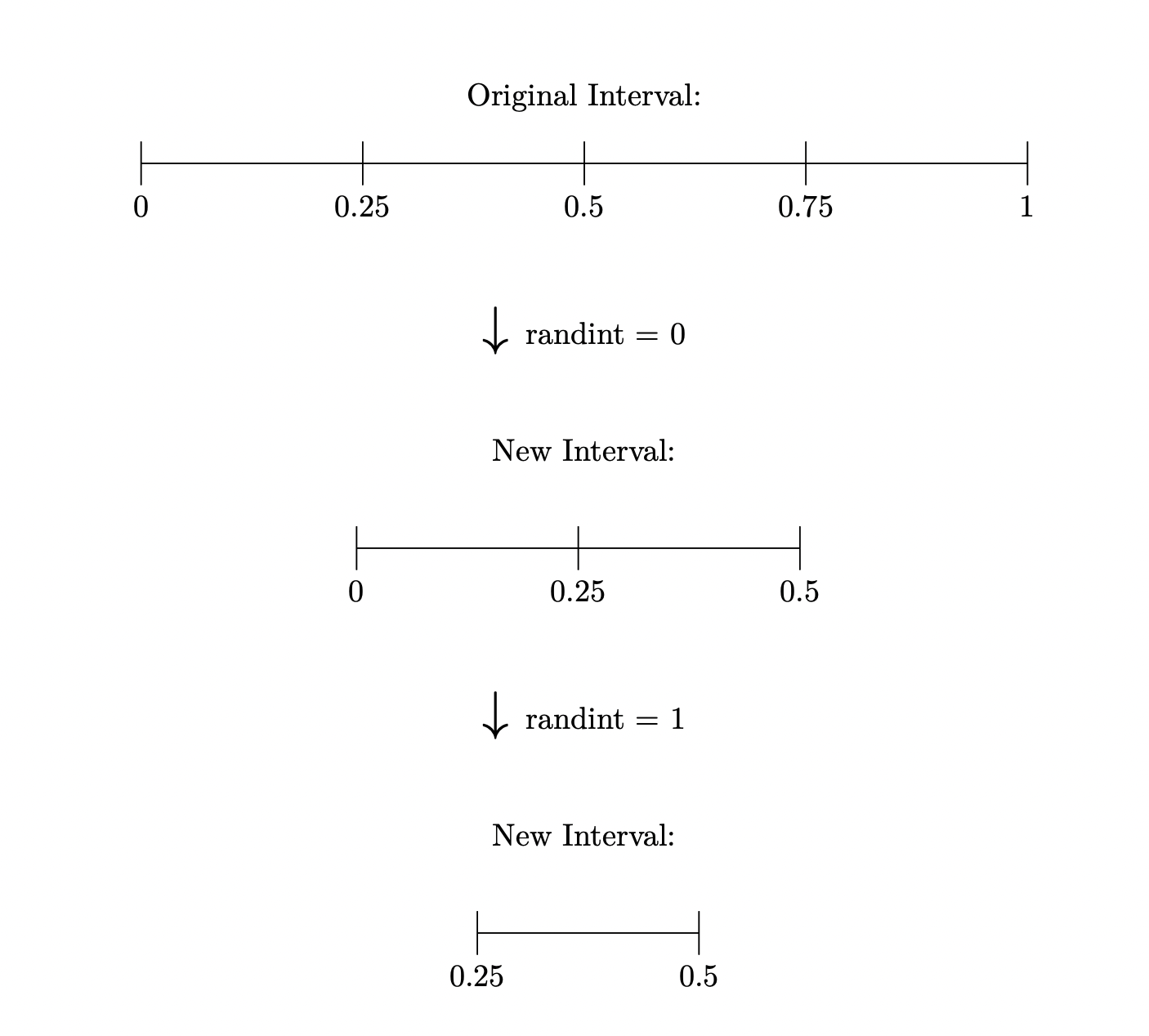}
    \caption{Representation of the Monte Carlo Bisection Algorithm}
\end{figure}

\subsection{Monte Carlo Bisection on a Brownian Bridge and Cauchy Process Path}
We now apply this Monte Carlo Bisection (MCB) method on a Brownian Bridge. In the following visualization, the blue dots represent approximated minima (that are appended to the ``minimum list'' each iteration) across $[0,1]$ and every pair of two consecutive red dots represent the randomly selected intervals over each iteration. Specifically, the following visualization is run with the parameters $l = 10, r = 10, g = 1024$ with a comparison graph to show the estimated global minimum compared with the estimated global minimum. 

\pagebreak
\begin{figure}[h]
\centering
\begin{subfigure}{.5\textwidth}
  \centering
  \includegraphics[width=5cm]{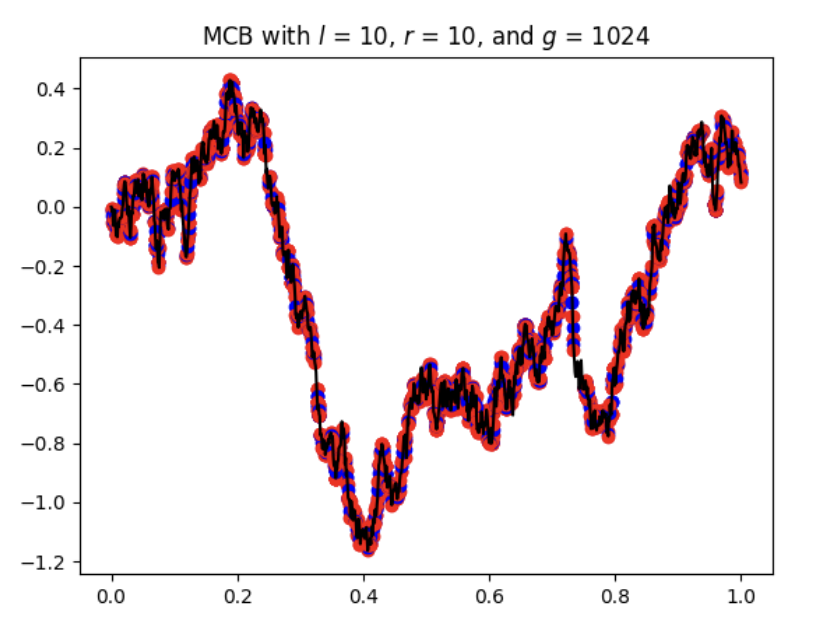}
  \caption{Visual graph of the MCB method}
  \label{fig:sub1}
\end{subfigure}%
\begin{subfigure}{.5\textwidth}
  \centering
  \includegraphics[width=5cm]{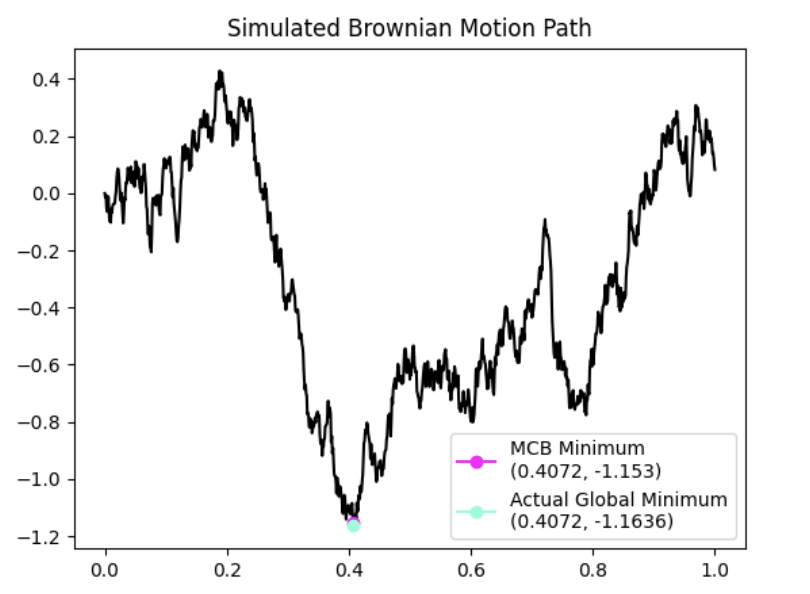}
  \caption{Actual global minimum vs. MCB estimation}
  \label{fig:sub2}
\end{subfigure}
\label{fig:test}
\caption{The MCB Method on a Brownian Bridge}
\end{figure}

As mentioned this method may also be utilized on a Cauchy process path. We will not be utilizing a Cauchy bridge, since this would require us to find the inverse of
\begin{equation}
G(u,v) = \int\frac{f_1(u+v)f_1(u-v)}{f_2(2u)}dv
\end{equation}
where $f_t(x) = \frac{1}{\pi} \cdot \frac{t}{t^2 + x^2}$, the probability density function of the Cauchy process. Of course, we can simplify this $G(u,v)$, but we will soon see why it is unrealistic to use the bridge. Since we want the integral over the whole number line to be equal to 1 for the integrand of $G(u,v)$, we want to find a constant of integration $C$ such that $\lim_{v\to-\infty} G(u,v) = 0$ and $\lim_{v\to\infty} G(u,v) = 1$. Eventually, we find $C = 1/2$ and 
\begin{equation}
\label{eq:Bayes}
\small G(u,v)=\frac{1}{4u\pi}\biggl(-\ln(u^2-2uv+v^2+1)+\ln(u^2+2uv+v^2+1)-2u\arctan(u-v)+2u\arctan(u+v)\biggr)+\frac{1}{2}
\end{equation}
Utilizing the inverse of $G(u,v)$ to simulate the Cauchy process will significantly slow down the program; thus, we simply use a regular Cauchy path to execute our algorithm. 
(But it might be useful in a hybrid algorithm we describe in Section 5.)
To simulate a Cauchy path, we utilize the Python package ``stochastic''\footnote{\url{https://github.com/crflynn/stochastic}} in order to simulate $2^l + 1$ points, allowing us to create a very similar algorithm (with the same defined parameters) to perform the MCB on the Cauchy process. A visual representation with $l = 10, r = 10, g = 1024$ and a comparison graph to show the global minimum compared with the estimated global minimum are shown below. 

\begin{figure}[h]
\centering
\begin{subfigure}{.5\textwidth}
  \centering
  \includegraphics[width=5cm]{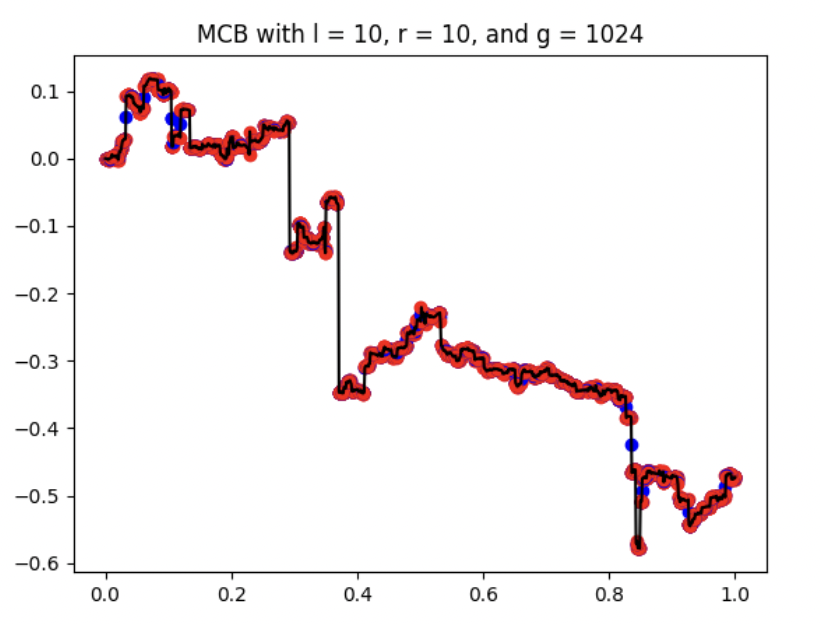}
  \caption{Visual for Cauchy MCB}
  \label{fig:sub1}
\end{subfigure}%
\begin{subfigure}{.5\textwidth}
  \centering
  \includegraphics[width=5cm]{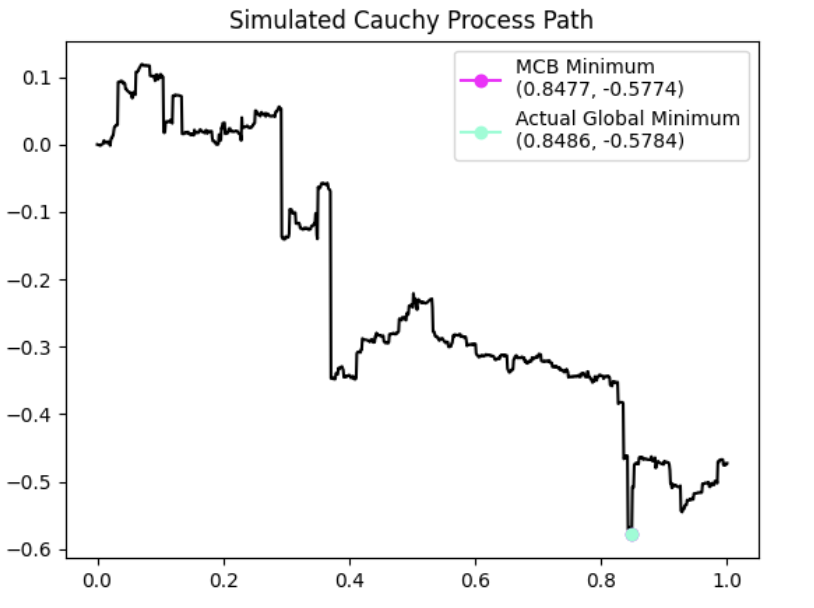}
  \caption{Actual global minimum vs. MCB estimation}
  \label{fig:sub2}
\end{subfigure}
\label{fig:test}
\caption{The MCB Method on a Cauchy Process Path}
\end{figure}

\section{Comparison Between Search Methods} \label{sec:3}
\subsection{Accuracy for Iterative GSS vs. MCB}
To test for the accuracy of both methods, we employ a similar method as to when we compared the Iterative GSS and the Naive Approach in Subsection \ref{subsection2}. For the Iterative GSS, we will once again run 500 iterations over $2^n$ partitions for each $n\in\mathbb{N}$ and $0<n<9$ and calculate the average error ($|$actual minimum - estimated minimum$|$) for each number of partitions. We will once again set $\varepsilon = 0.001$. However, for the MCB method, we will run 500 iterations for $l = n, r = n, g = 2^n$, where $n\in\mathbb{N}$ and $0<n<15$ since the program has higher computational efficiency to handle these large parameters, which will be seen in the following subsection. For the MCB, we will also calculate the average error for each parameter combination, and all of these simulations will be run on a Brownian Bridge, since the Cauchy process does not scale well for the Iterative method. The results are shown in the table below.

\begin{table}[h]
\centering
\caption{Accuracy Comparison for the Iterative and MCB Methods}
\begin{subfigure}{.5\textwidth}
  \centering
  \includegraphics[width=8cm]{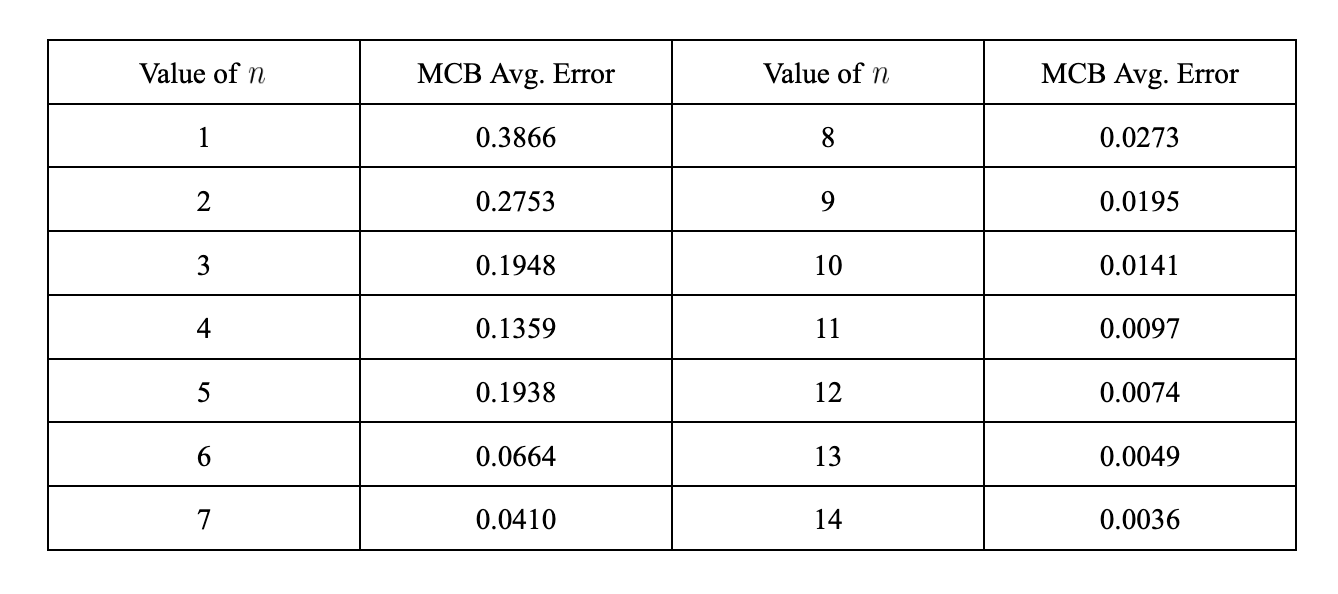}
  \label{fig:sub1}
\end{subfigure}%
\begin{subfigure}{.5\textwidth}
  \centering
  \includegraphics[width=7.4cm]{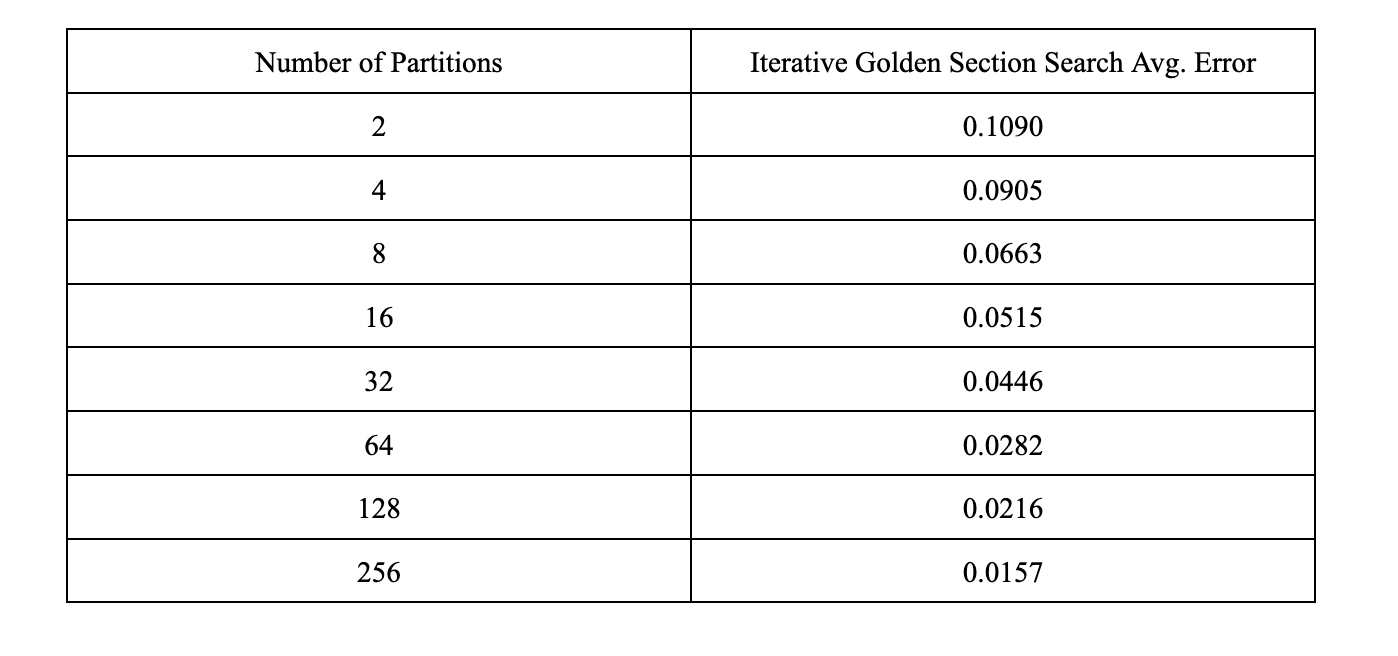}
  \label{fig:sub2}
\end{subfigure}
\end{table}

\subsection{Efficiency for Iterative GSS vs. MCB}
To test for the efficiency of both methods, we will disregard the actual accuracy, but focus on the time it takes for each algorithm to run. We do this by utilizing the time module in Python to track the run times of each algorithm. We will be using the same parameters as specified in the previous subsection, and the results are shown in the table below. 

\begin{table}[h]
\centering
\caption{Run Time Comparison for the Iterative and MCB Methods}
\begin{subfigure}{.5\textwidth}
  \centering
  \includegraphics[width=8cm]{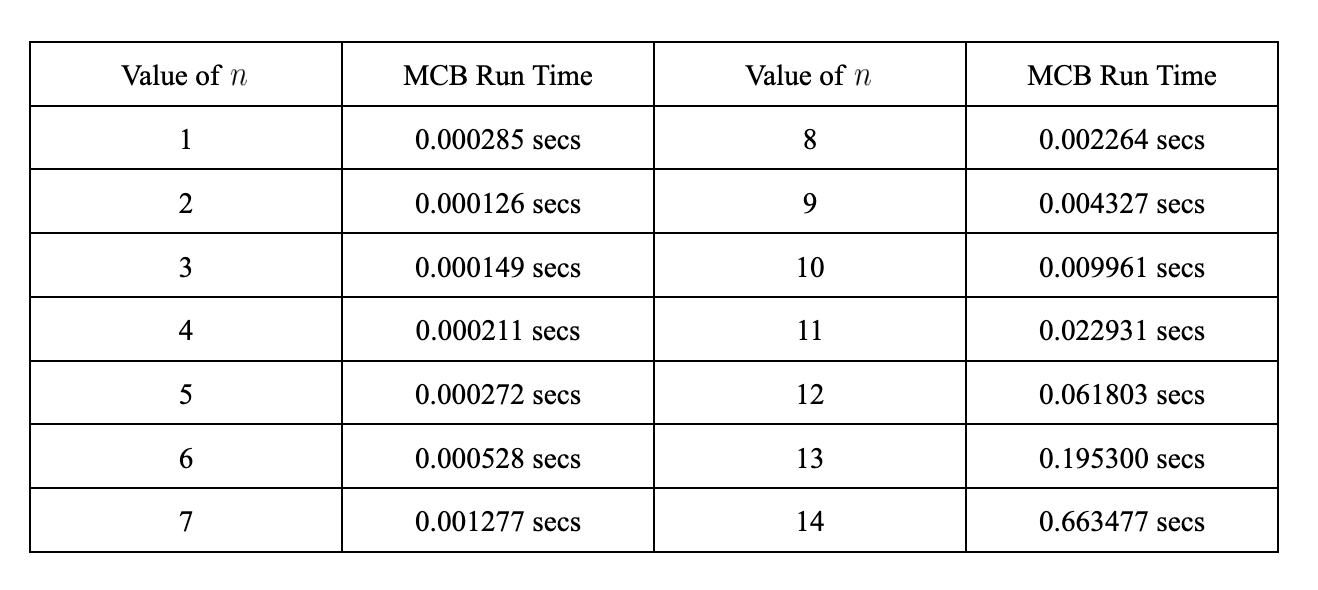}
  \label{fig:sub1}
\end{subfigure}%
\begin{subfigure}{.5\textwidth}
  \centering
  \includegraphics[width=7.6cm]{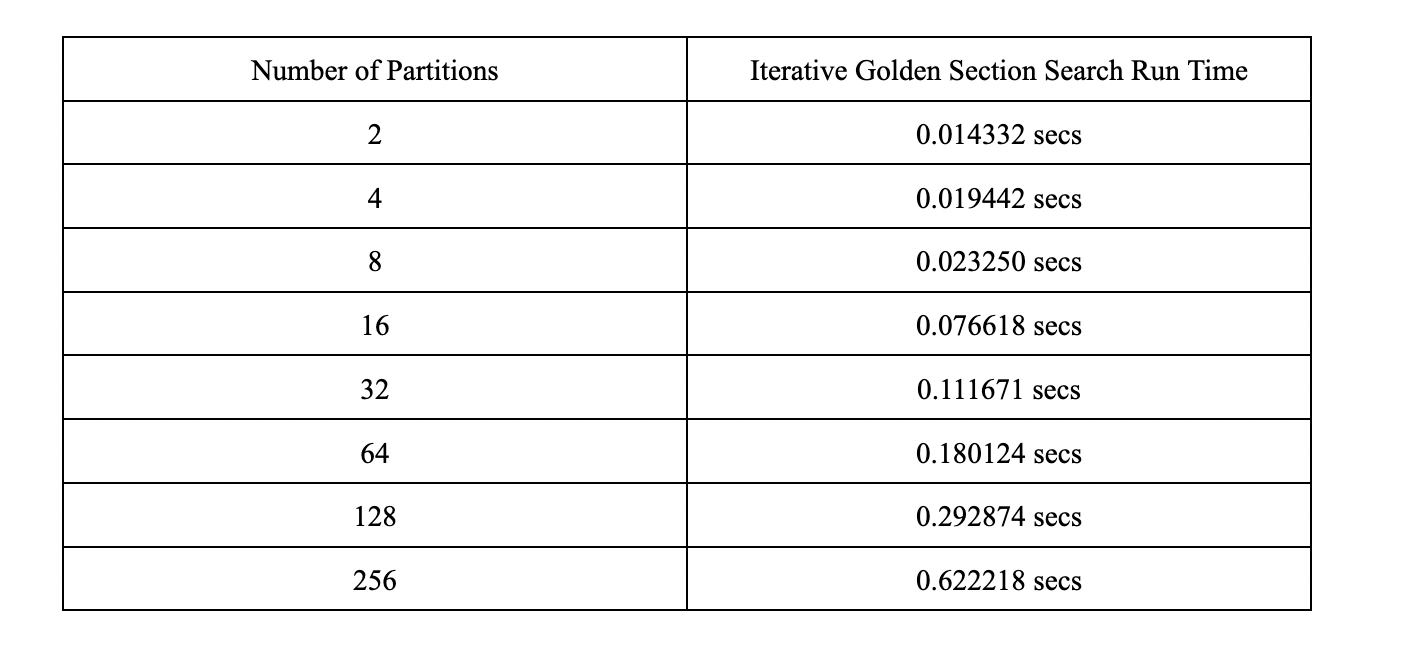}
  \label{fig:sub2}
\end{subfigure}
\end{table}

\section{Conformal mappings and the Harmonic measure } \label{sec:4}

Suppose that we have an open, connected domain in $\mathbbm{CP}^1 = \C \cup \{\infty\}$, called 
$\mathcal{U}$.
Suppose that the boundary $\partial \mathcal{U}$ is a  Jordan curve.
Given a point $z \in \mathcal{U}$, there exists a conformal mapping 
\begin{equation}
	\varphi : U(0;1) \to \mathcal{U}\, ,
\end{equation}
which is a bijection, and 
such that $\varphi(0) = z$. Moreover, the mapping is unique modulo pre-compositions with rotations.

Given all of this, one may define a probability measure $\omega(\cdot,\mathcal{U},z)$ on the Borel measurable subsets
$E$ 
of $\partial \mathcal{U}$ by pulling back to the rotation invariant uniform measure on the unit circle:
\begin{equation}
\label{eq:Invariance}
	\omega(E,\mathcal{U},\varphi(0))\, =\, \frac{1}{2\pi}\, |\varphi^{-1}(E)|\, 
=\, \frac{1}{2\pi} \int_0^{2\pi} \mathbf{1}_E(\varphi(e^{i\theta}))\, d\theta\, .
\end{equation}
This is called the Harmonic measure. 

Under general conditions there are two other equivalent formulations. Firstly, 
given a suitable boundary function $f : \partial U \to \R$, consider
the Laplace
equation for a function $\psi : \mathcal{U} \to \R$ with Dirichlet-type boundary conditions,
\begin{equation}
	\begin{cases} \Delta \psi\, =\, 0\, ,\ & \text{ in $\mathcal{U}$,}\\
\psi\, =\, f\, ,\ & \text{ on $\partial \mathcal{U}$.}
\end{cases}
\end{equation}
The unique solution is given by the formula
\begin{equation} 
	\psi(z)\, =\, \int_{\partial \mathcal{U}} f(\zeta)\, d\omega(\zeta,\mathcal{U},z)\, ,
\end{equation}
for all $z \in \mathcal{U}$.
This is the source of the name, ``Harmonic measure,'' to describe 
$\omega(\cdot,\mathcal{U},z)$.

The other equivalent formulation is as follows.
We can start a 2d Brownian motion at the point $z=x+iy$, viewing the 2d Brownian motion
as running on $\C$.
Then we stop the Brownian motion whenever it hits $\partial \mathcal{U}$.
The hitting time is 
\begin{equation}
	\tau(\partial \mathcal{U})\, =\, \min(\{t :\, B_t \in \partial \mathcal{U}\})\, .
\end{equation}
Then we have the formula for the Harmonic measure:
\begin{equation}
	\omega(E,\mathcal{U},z)\, =\, \mathbf{P}(\{B_{\tau(\partial \mathcal{U})} \in E\}\, |\, \{B_0 = z\})\, ,
\end{equation}
for any Borel measurable subset $E \subseteq \partial \mathcal{U}$.
That is our interest, here, as we explain next.

\subsection{Hitting distribution on a bridge for reflected B.m.\ on a strip}

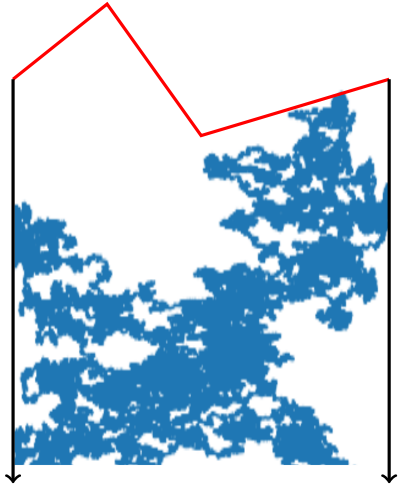
\begin{figure}
\begin{center}
\begin{tikzpicture}[xscale=5,yscale=1]
	\draw (0.505,-3.25) node[yscale=-1] {\includegraphics[width=5.1cm,height=6.625cm]{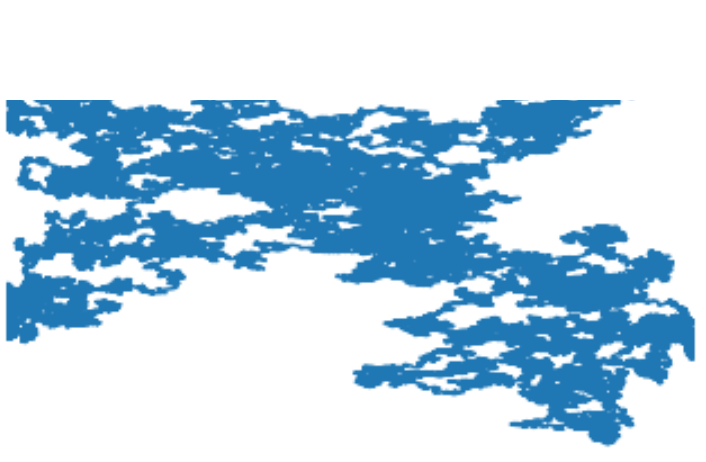}};
	\draw[very thick,red] (0,0) -- (0.25,1) -- (0.5,-0.75) -- (0.75,-0.375) -- (1,0);
	\draw[very thick,->] (0,0) -- (0,-5.375);
	\draw[very thick,->] (1,0) -- (1,-5.375);
\end{tikzpicture}
\end{center}
\caption{This is a schematic view of a walk, colored red, and the domain $\mathcal{U}$
consisting of the semi-infinite strip beneath the walk. We are interested in a reflected Brownian
motion diffusing up from $-\infty i$, colored blue. It is is stopped at the first time it hits the red walk.
We propose the hitting distribution as a guideline to the next bisection edge.
\label{fig:walk}
}
\end{figure}

Suppose that we have a walk of length $n$: a sequence of ordered pairs $(t_1,W_1)$, \dots, $(t_n,W_n)$ 
satisfying
$0=t_0<t_1<\dots<t_n=1$ and $W_1,\dots,W_n \in \R$ with the extra definition $W_0=0$.
Let us also assume that $W_n=0$ in order to have a bridge.
Define the piecewise linear interpolating function $\gamma : [0,1] \to \R$ 
using the formula 
\begin{equation}
\gamma(t)\, =\, \sum_{k=1}^n\mathbf{1}_{(t_{k-1},t_k]}(t) \cdot \left(\frac{t_k-t}{t_k-t_{k-1}}\,  W_{k-1} 
+ \frac{t-t_{k-1}}{t_k-t_{k-1}} W_{k}\right)\, .
\end{equation}
Then let us denote an open domain in $\C$ to be
\begin{equation}
	\mathcal{U}\, =\, \{x+iy\, :\, 0<x<1\, ,\ y<\gamma(x)\}\, .
\end{equation}
Let $B_t$ be a 2d Brownian motion. 
In order to constrain it to the strip where $0\leq x\leq 1$, let us apply the reflection
\begin{equation}
	\mathcal{R}(B_t)\, =\, \pi^{-1} \cos^{-1}\left(\pi \cos(\mathrm{Re}(B_t))\right) + i \mathrm{Im}(B_t)\, ,
\end{equation}
where $\operatorname{Re}(B_t) = (B_t + \overline{B}_t)/2$ and $\operatorname{Im}(B_t) = (B_t - \overline{B}_t)/(2i)$
are the usual real and imaginary parts.
Then $\mathcal{R}(B_t)$ is reflected Brownian motion constrained to the strip.
Let us define the top portion of $\partial \mathcal{U}$ as  $\Gamma$. Hence,
\begin{equation}
\Gamma\, =\, \{x+iy\, :\, 0\leq x\leq 1\, :\, y=\gamma(x)\}\, .
\end{equation}
Let us define a hitting time $\widetilde{\tau}(\Gamma)$ as
\begin{equation}
	\widetilde{\tau}(\Gamma)\, =\, \min(\{t\, :\, \mathcal{R}(B_t) \in \Gamma\})\, .
\end{equation}
Then, for each $z \in \mathcal{U}$, let us define $\widetilde{\omega}(\cdot,\Gamma,z)$ to be the measure on
Borel subsets of $\Gamma$ using the formula
\begin{equation}
	\widetilde{\omega}(E,\Gamma,z)\, =\, \mathbf{P}(\{\mathcal{R}(B_{\widetilde{\tau}(\Gamma)}) \in E\}\, |\, \{B_0=z\})\, .
\end{equation}
Given a boundary value function $f : \Gamma \to \R$ the solution of the Laplace equation 
for $\psi : \mathcal{U} \to \R$, with mixed boundary conditions
\begin{equation}
\begin{cases} \Delta \psi\, =\, 0\, ,\ & \text{ in $\mathcal{U}$,}\\
\partial \psi/\partial n\, =\, 0\, ,\ & \text{ on $\{iy\, :\, y<0\}$,}\\
\psi\, =\, f\, ,\ & \text{ on $\Gamma$,}\\
\partial \psi/\partial n\, =\, 0\, ,\ & \text{ on  $\{1+iy\, :\, y<0\}$,}
\end{cases}
\end{equation}
is given by $\psi(z) = \int_{\Gamma} f(\zeta)\, d\widetilde{\omega}(\zeta,\Gamma,z)$ for $z \in \mathcal{U}$.
This is a property of reflected Brownian motion which could also be seen using the ``method of images.''

We are primarily interested in the limit of $\widetilde{\omega}(\cdot,\Gamma,x+iy)$ as $y \to -\infty$,
 for any $x \in (0,1)$.
Let $\mathbb{H}^+$ denote the domain $\{x+iy\, :\, x \in \R\, ,\ y>0\}$.
Note that its  boundary is $\R \cup \{\infty\}$.
\begin{theorem}
\label{thm:beta}
Consider the unique conformal mapping $\varphi : \mathbb{H}^+ \to \mathcal{U}$
which is a bijection and such that
for the continuous extension to the boundaries, the homeomorphism \footnote{This exists by Caratheodory's theorem. See for example \cite{Krantz}.} satisfies: $\infty$ is mapped to $\infty$
and $[0,1]$ is mapped bijectively onto $\Gamma$.
Then $\widetilde{\omega}(\cdot,\Gamma,z)$ converges weakly as $|z| \to \infty$ to the measure $\widetilde{\omega}(\cdot,\Gamma,\infty)$ given by the formula
\begin{equation}
\label{eq:thm1}
	\widetilde{\omega}(E,\Gamma,\infty)\, =\, \frac{1}{\pi}\, \int_0^1 \frac{\mathbf{1}_E(\varphi(x))}{\sqrt{x(1-x)}}\, dx\, .
\end{equation}
\end{theorem}
\begin{proof}
Suppose we have a function $f : \Gamma \to \R$. Consider the function $u : \C \setminus [0,1]$ such that $u(z)$
is the expectation of $f(\varphi(X_{\tau[0,1])}))$ where $X_0$ starts at $z$. Then $u$ is Harmonic on $\C \setminus [0,1]$
and for any $x \in [0,1]$ we have $\lim_{z \to x} u(z) = f(\varphi(x))$. Moreover, by symmetry, the normal derivative of $u$ is zero along $(-\infty,0)$
and $(1,\infty)$. 
More precisely, $u(z) = u(\overline{z})$ by symmetry of the geometry.
Since $(-\infty,0)$ and $(1,\infty)$ are mapped to the left and right
edges of $\mathcal{U}$ under $\varphi$ we see that the normal derivative is preserved. So taking $\psi = u \circ \varphi^{-1}$,
it solves the defining PDE.
Finally, it is easy to calculate the Harmonic measure on $[0,1]$ as the boundary of $\C \setminus [0,1]$, for any starting point $z \in \C \setminus [0,1]$ by using a conformal mapping
between the interior of the unit circle and the exterior of the unit interval (a 2-to-1 mapping of the boundaries everywhere except the endpoints).
The $z \to \infty$ limit is the arcsine law.
\end{proof}

Our proposal is the following. Consider a function $g : [0,1] \to \R$ such that $g(0)=g(1)=0$.
Let us seek the minimum value of $g$.
We assume that evaluating $g$ at a point $x$ is computationally expensive.
We will use the bisection method, but we choose exactly 1 sub-interval constructed so far
to bisect.
The strategy for choosing is one of the following.
\begin{itemize}
\item {\em Strategy 1.} We calculate the measure $\omega(E_k,w,\infty)$ for each edge $E_k \in \Gamma$, where the points $(t_k,W_k)$
are the values of the graph of $g$ revealed so far: $W_k = g(t_k)$.
Then $w$ is the piecewise linear interpolation, as above, and $\Gamma$ is the curve associated to the graph.
$E_k$ is the line segment with endpoint $t_{k-1}+i W_{k-1}$ and $t_k + i W_k$.
Then we take the interval with the maximum measure.
\item {\em Strategy 2.} We use a Monte Carlo algorithm. We randomly choose interval $E_k$ with probability
$\omega(E_k,w,\infty)$.
\end{itemize}

The heuristic is that the reflected Brownian motion diffusing up from $-\infty i$ does a good job of predicting
the best interval to bisect at the next stage, based on the geometry already revealed.
We note that one could rescale the function $g$, replacing it with $\beta g$ for some $\beta \in [0,\infty)$.
Then that would alter the strategies. So $\beta$ could be considered a thermodynamic parameter
such as inverse-temperature, as occurs in Monte Carlo Markov chains.
%

\subsection{The Schwarz-Christoffel formula and SC toolbox}

In complex analysis a polygon on $\mathbb{CP}^1$ is a simple closed curve which is piecewise linear (in the standard geometry
of $\C \cup \{\infty\}$) with only finitely many vertices: $w_1,\dots,w_m$. One of the vertices may be $\infty$, in which case two of the edges
will be infinitely long. An example of such a polygon is given in Figure \ref{fig:walk}, with the walk and the semi-infinite strip below the walk being a polygon with one vertex at $\infty$ and 5 vertices.

For the domain $\mathcal{U}$,  the interior of the polygon, the conformal mapping from $\mathbb{H}^+$ to $\mathcal{U}$, 
may be formulated semi-explicitly by the Schwarz-Christoffel formula:
for some $A \in \C$ and $C \in \C \setminus \{0\}$,
\begin{equation}
	\varphi(z)\, =\, A + C \int^z \prod_{k=1}^{m-1} (\zeta - z_k)^{\alpha_k-1}\, d\zeta\, ,
\end{equation}
assuming that $\varphi(\infty) = w_m$. We take this formulation from Driscoll and Trefethen \cite{DriscollTrefethen}.
The numbers $z_1,\dots,z_n \in \R \cup \infty$ are the pre-vertices $\varphi(z_k) = w_k$ in
the continuous extension of $\varphi$ to the boundary.
The internal angles of the vertices are $\alpha_1 \pi,\dots,\alpha_n \pi$ satisfying $\alpha_k \in (0,2]$ for finite
vertices. If one of the vertices is $\infty$ then that interior angle satisfies $\alpha_k \in [-2,0]$.
This is semi-explicit because the pre-vertices must be determined.

As one can see in Figure \ref{fig:walk} our natural enumeration of the vertices is clockwise;
whereas, typically in mathematics one enumerates vertices in counterclockwise order.
Because of this, we need to use $\mathbb{H}^- = \{x+iy\, :\, x \in \R\, ,\ y<0\}$.
The Schwarz-Christoffel formula can still be written as 
\begin{equation}
	\varphi(z)\, =\, A + C \int^z \prod_{k=1}^{m-1} (\zeta - z_k)^{\alpha_k-1}\, d\zeta\, ,
\end{equation}
but where we use a branch of the logarithm going along the positive imaginary axis
such that we say $\arg(x)=0$ for $x>0$.
This forces $\arg(-x)=-\pi$ for $x>0$.
\begin{corollary}
\label{cor:arcsine}
Let $(t_1,W_1)$ to $(t_n,W_n)$ be as above.
Consider the polygonal boundary of $\mathcal{U}$ with vertices $w_1,\dots,w_{n+1},w_{n+2}$
(so $m=n+2$)
where $w_k = t_{k-1} + i W_{k-1}$ for $k=1,\dots,n+1$ and $w_{n+2} = \infty$.
Then for the Schwarz-Christoffel mapping from $\mathbb{H}^- \to \mathcal{U}$
such that $0=z_1<\dots<z_{n+1}=1$, we have for each of the edges $E_k$, $k=1,\dots,n$, joining $t_{k-1}+iW_{k-1}$
to $t_k+iW_k$,
\begin{equation}
	\widetilde{\omega}(E_k,\Gamma,\infty)\, =\, \sin^{-1}\left(\sqrt{z_{k+1}}\right) 
- \sin^{-1}\left(\sqrt{z_{k}}\right)\, .
\end{equation}
Thus, in our strategies, the necessary inputs are given just by the determination of the pre-vertices.
\end{corollary}
\begin{proof}
Use Theorem \ref{thm:beta}, and the well-known ``arcsine law.''
\end{proof}

The basic output of the SC toolbox is the list of pre-vertices for the conformal mapping of $U(0;1)$
to $\mathcal{U}$, where $U(0;1)$ is the open unit disk in $\C$, centered at $0$.
For this purpose, let us note a conformal mapping
$\Phi : U(0;1) \to \mathbb{H}^-$ is given by the formula
\begin{equation}
	\Phi(z)\, =\, -i\, \frac{1-z}{1+z}\, .
\end{equation}
If we labeled the prevertices as $Z_1,\dots,Z_{n+1},Z_{n+2}$ where $Z_{n+2}$ is the pre-vertex of $\infty$,
then we may first rotate the disk to insure that $Z_{n+2} = -1$,
which we assume was already done, without loss of generality.
Then applying $\xi_k=\Phi(Z_k)$, and then taking $z_k=a+c\xi_k$ for some $a \in \R$ and $c>0$ we have prevertices on $\R$ with $0=z_1<z_2<\dots<z_{n+1}=1$ and $z_{n+2}=\infty$.
This is what is needed in Corollary \ref{cor:arcsine}.

By this result, we can see the high value of the SC toolbox by Driscoll and Trefethen,
maintained by Driscoll, whose purpose
is exactly to calculate the pre-vertices.
It is written for Matlab.
\begin{figure}
\begin{center}
\raisebox{5cm}{\begin{minipage}{6cm}
\includegraphics[width=8cm,height=9cm]{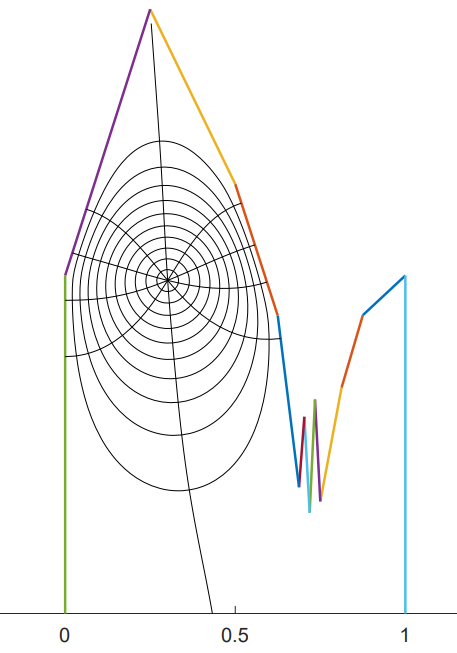}
\end{minipage}}
\hspace{2cm}
\includegraphics[width=7cm,height=11cm]{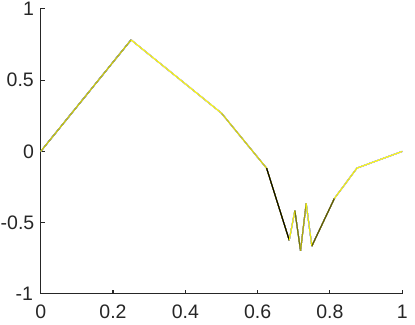}
\caption{This gives two perspectives on the same walk. On the left is the output of the SC toolbox for the conformal map from a disk, indicated by showing the image of the polar coordinate system. (The aspect ratio is not necessarily true.) For the picture on the left, the conformal center is chosen arbitrarily by the software package. On the right is the walk with the color map interpolating from yellow to black using the Harmonic measure of each edge to determine its color.
For the minimum Harmonic measure of an edge, that edge is colored yellow. For the maximum Harmonic measure of an edge, that edge is colored black. All other colors
are assigned according to the Harmonic measure, similarly.
\label{fig:scColor}
}
\end{center}
\end{figure}
In Figure \ref{fig:scColor} we show the output for a typical walk.
Note that if the points of the walk are $p_k = t_k + i \gamma(t_k) = t_k + i W_k$, then the method is as follows.
Because the SC toolbox requires angles as an input whenever one of the coordinates is $\infty$, we do as follows.
Let 
\begin{equation}
	\left(\frac{\Delta \gamma}{\Delta t}\right)_{k+\frac{1}{2}}\, =\, \frac{\gamma(t_{k+1}) - \gamma(t_k)}{t_{k+1}-t_k}\,
	=\, \frac{W_{k+1}-W_k}{t_{k+1}-t_k}\, =\, \frac{\mathrm{Im}[p_{k+1}-p_k]}{\mathrm{Re}[p_{k+1}-p_k]}\, .
\end{equation}
Then we may define $\eta_k^+$ and $\eta_{k+1}^-$ by the formula
\begin{equation}
	\pi \eta_k^+\, =\, \frac{\pi}{2} + \tan^{-1}\left(\left(\frac{\Delta \gamma}{\Delta t}\right)_{k+\frac{1}{2}}\right)\qquad
\text{ and }\qquad
	\pi \eta_{k+1}^-\, =\, \frac{\pi}{2} - \tan^{-1}\left(\left(\frac{\Delta \gamma}{\Delta t}\right)_{k+\frac{1}{2}}\right)\, ,
\end{equation}
as long as $k$ and $k+1$ are both in $0,\dots,n$.
Let us also define $\eta_0^-=\eta_n^+=0$.
Then we define $\alpha_k$ for $k \in \{1,\dots,n+1\}$ to be
\begin{equation}
	\alpha_k\, =\, \eta_{k-1}^+ + \eta_{k-1}^-\, ,
\end{equation}
where the mis-match in indices is due to the definitions in SC toolbox (which uses Matlab, whose indexing system
must always start with $1$ rather than $0$).
The angle for $\infty$ is usually irrelevant in the Schwarz-Christoffel formula, and in this case it will be $0$.
Note that
\begin{equation}
\sum_{k=0}^{n} (\eta_k^+ + \eta_k^-)\, =\, \sum_{k=1}^{n} \eta_k^- + \sum_{k=0}^{n-1} \eta_k^+\, =\, n\, ;
\end{equation}
whereas, there are a total of $n+2$ points. So $\sum_{k=1}^{n+2} (1-\alpha_k) = 2$, as is always required in the 
Schwarz-Christoffel formula. 
\begin{figure}
\begin{center}
\includegraphics[width=8cm]{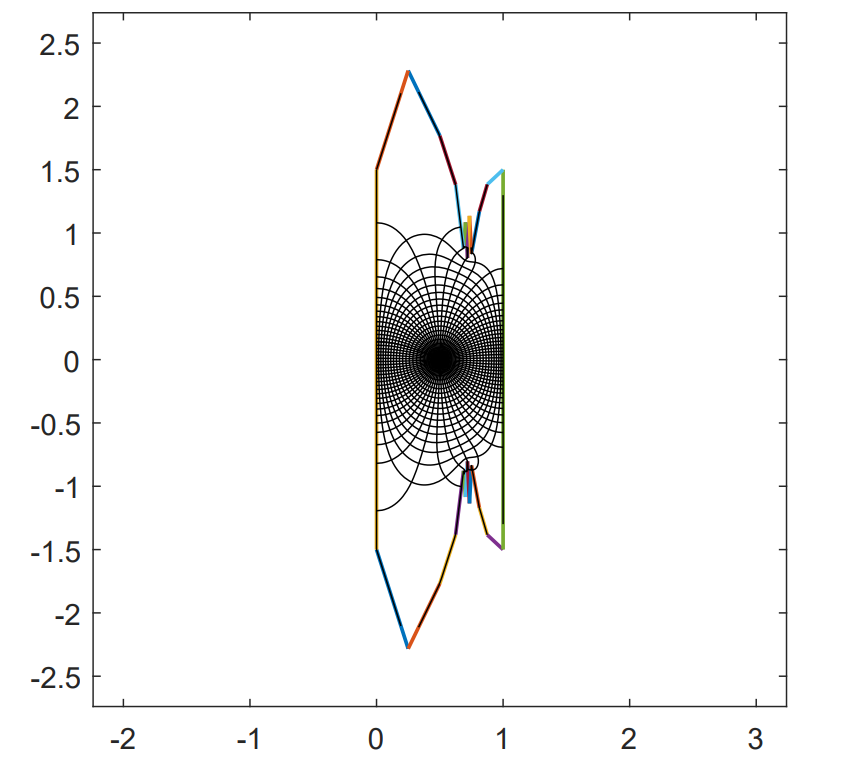}
\caption{This gives another perspective of the walk. We artificially moved the conformal center to a point of our own choosing by
creating a polygon with top-to-bottom reflection symmetry, instead of the more legitimate method wherein one of the coordinates
is at $-\infty$. This way, we do not need to calculate the angles, since the SC toolbox can calculate the angles itself if all vertices
are finite.
\label{fig:scCenter}
}
\end{center}
\end{figure}

\section{Applications and Outlook} \label{sec:5}

The main goal of considering the problem of online optimization of Brownian motion is pedagogical.
It serves as a toy model for more sophisticated problems of optimizing functions in higher dimensional
vector spaces.
However, there are also applications of this 1-dimensional problem directly.

In finance, any contingency claim based on a single stock is supposedly priced correctly according to the 
Black-Scholes-Mertons formula \cite{Hull}. As an example, a Cash-or-Nothing call option allows for a simple
bet that a stock price will exceed a certain price at a certain date (and time).
But in an American-style version one might just wish to make a bet that the stock price ever exceeds a given
strike price at any time from present to the expiration date. If so, a certain payoff may be exchanged at the 
expiration date. If one adjusts the stock price for interest rate, and volatility, this amounts to determining
whether the maximum of Brownian motion will pass a certain value.
One could imagine trying to price this option for a stock by querying a computerized library of all stock prices.

Another proxy for these types of problem is determining the distribution of the-maximum-minus-the-minimum
of Brownian motion (or a Brownian bridge), and its joint distribution  with the distance between the two times
where the maximum and minimum occur.
This is a reasonable question, but one for which an exact answer seems to be not forthcoming.
For the joint distribution of the maximum and the time of the maximum, a formula is known
for both the Brownian motion and the bridge using Girsanov's formula. See for example, Karatzas and Shreve 
\cite{KaratzasShreve}.
But for the quantity we mentioned, all that is known is a series formula equivalent to elliptic theta functions.
The joint distribution seems to not be known.
See for example \cite{ChoiRoh}.

This could be done numerically quite simply using the algorithms described herein.
(It is our intention to perform these calculations and post them as supplementary materials in the near future.)
Note that Mandelbrot also promoted the Cauchy process as a possibly more realistic alternative to Brownian motion
for certain financial applications. See for instance \cite{Mandelbrot}.
He viewed Brownian motion as too tame to adequately reflect financial movements.
But the Cauchy process may be too wild.

The Cauchy process is applicable in other models in science. 
For example, it is proposed as a model of phylogenetic fitness in branching models, compared
to real world data.
See Bastide and Didier \cite{BastideDidier}.
There questions about the maximum fitness over a particular evolutionary path
is relevant.
We note that the distribution of the maximum of a Cauchy process is solvable using the reflection principle.
But Girsanov's method is not as simple to apply. So the distribution of the maximum of a Cauchy bridge
may not be known.
This could be easily calculated numerically using our methods.
(It is our intention to perform these calculations and post them as supplementary materials in the near future.)

Finally, as an open question, we mention that the harmonic Monte Carlo algorithm could be extended to higher dimensions.
Harmonic measure in $d$-dimensions is a well-studied topic, which is natural, unlike conformal mappings
in higher dimensions.
We would like to explore this in the future for real world problems.
Of course, calculation of the Harmonic measure is computationally somewhat expensive, itself.
But for sufficiently complex problems, it may be less expensive than the cost savings in making fewer queries
of the function.

\subsection{A hybrid online search model}

\begin{figure}
    \centering
    \includegraphics[height=4cm]{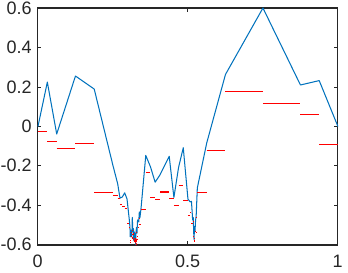}\hspace{1cm}
    \includegraphics[height=4cm]{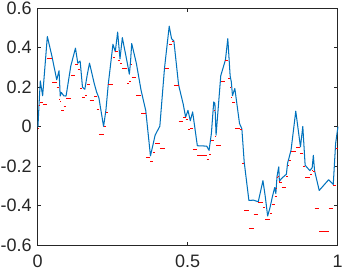}
    \caption{On the left we have output the sample points for a Brownian bridge according to our hybrid algorithm with $\beta=1$ and $n=100$. The red horizontal line segments
    are simulations of the minimum on each interval. 
    On the right we show the output where $\beta=0$,
    which is a proxy for a purely random sample of the endpoints (or for example, choosing the times along a lattice).}
    \label{fig:BBMS2}
\end{figure}
Motivated by the Harmonic Monte-Carlo scheme described above, we consider another algorithm to seek the minimum of a Brownian bridge.
\begin{itemize}
\item Start with the values at time $0$ and $1$, namely $X_0=X_1=0$.
\item Use the Brownian bridge distribution to sample a value at the midpoint $X_{1/2}$.
\item At the $n$th stage, when we have a list of times $0=t_0<t_1<\dots<t_n=1$, calculate the median values for $\min(\{X_t\, :\, t_{k-1}\leq t\leq t_k\})$ for each $k=1,\dots,n$.
\begin{itemize}
    \item Let $a$ be the minimum of all these medians. (In principle we should seek the median of the minimum instead. But the minimum of the medians is easier to calculate, and we believe is a good stand-in
    for large iterations.)
    \item For each $k=1,\dots,n$ let $f_k$ be the probability that $\min(\{X_t\, :\, t_{k-1}\leq t\leq t_k\})$ is less than or equal to $a$.
    \item Form a probability mass function by taking $p_k = f_k^{\beta} (t_k-t_{k-1})/Z_n$ for a fixed, chosen $\beta \in [0,\infty)$, where $Z_n$ is a normalization constant.
    \item Select the next interval to bisect randomly according to the mass function $p_1,\dots,p_n$. Then repeat the algorithm at the next step.
\end{itemize}
\end{itemize}
It might seem more reasonable, instead of using $\min(\{X_t\, :\, t_{k-1}\leq t\leq t_k\})$ for several intermediate steps to just use $X_{(t_{k-1}+t_k)/2}$.
But it is analytically easier to deal with the minimum due to the simplicity arising from the reflection principle. (The cumulative distribution function of the minimum
is expressible in terms of the pdf of Brownian motion.)
The two should be close for small intervals, in any case, which is asymptotically what the algorithm will focus on for large iterations.
\begin{figure}
    \centering
    \includegraphics[height=4cm]{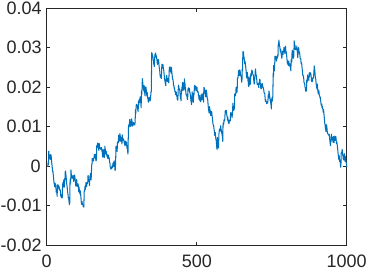}\hspace{1cm}
    \includegraphics[height=4cm]{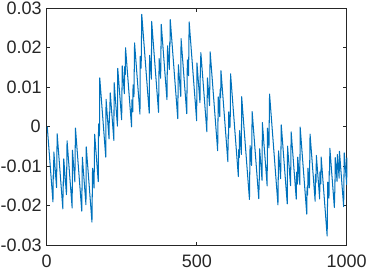}
    \caption{Here we show the Kolmogorov-Smirnov statistic
    for that distribution using our hybrid scheme. On the left is the result of 1000 samples for $\beta=1$ and $n=50$. It looks like a Brownian
    bridge, as is should according to the KS test. On the right is the same output for $\beta=0$ as a comparison.}
    \label{fig:KS}
\end{figure}
In Figure \ref{fig:BBMS2} we show the output for $\beta=1$, which is the value of the parameter we consider to be most natural.
As a benchmark, we also show a comparison plot for $\beta=0$, which is purely random selection of $t_k$ points
({\em a priori} measure is uniform).

We also show another perspective. We used our hybrid algorithm to sample the argmin's of the Brownian bridge. By the cycle method, it is well
known that this is a uniform distribution. See Chung and Feller's second theorem \cite{ChungFeller, Feller}. Another insightful perspective is 
by Gessel and Huq \cite{GesselHuq}.
In Figure \ref{fig:KS}, we plot the difference between our empirical distribution for 1000 samples, and the exact uniform distribution.
According to the Kolmogorov-Smirnov test this should be a Brownian bridge, itself (in the limit of infinitely many sample points). 
See, for example, Dudley, Proposition 12.3.2 \cite{Dudley}.

We can also carry out the same type of analysis for the Cauchy bridge. The reflection principle still applies, as it does to all symmetric
alpha-stable processes (or symmetric L\'evy processes, more generally). The harder part is calculating the Bayes formula for the conditional
distribution of a midpoint. If we assume that $X_2=x$ and write $u=x/2$, and $v=X_1-u$, then the conditional cumulative distribution function for $v$ is given by equation (\ref{eq:Bayes}).
As we stated before, we do not know a formula for the inverse of the CDF as would be needed to simulate this distribution using a uniform
random variable.
But if we let $U$ be a uniform random variable on $(0,1)$, and we use an equation solver (Newton method) to find the $v$
such that $G(u,v)=U$, then calling that value of $v$ by $V$, we see that $V$ has the proper distribution.
\begin{figure}
    \centering
    \includegraphics[height=4cm]{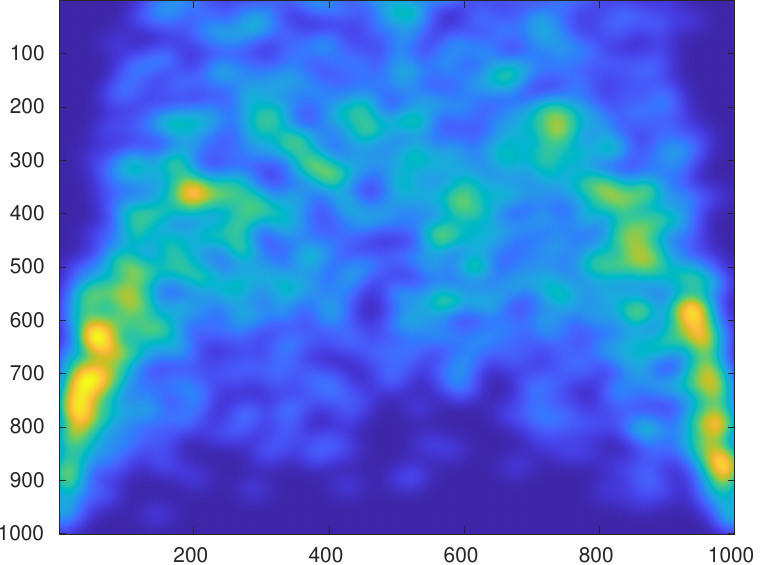}\hspace{1cm}
    \includegraphics[height=4cm]{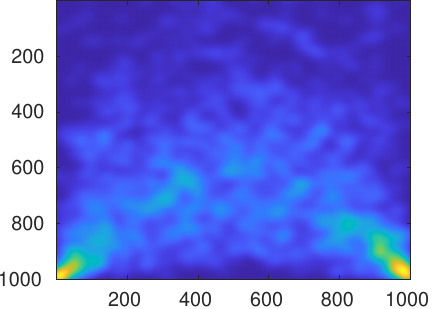}
    \caption{Here we plot the scatter plot ($N=1000$) of the coordinate pairs of argmin and the negative of the minimum value.
    On the left is the result of the Brownian bridge. On the right is the result for the Cauchy bridge. We used imagesc in Matlab to generate these pictures. (The brightness is not the same in the two pictures.) The $x$-values label the argmin, whose range is $0$ to $1$. The $y$-value labels the negative of the minimum value. For the $y$-values we restricted the range to $0$ to $1$.}
    \label{fig:scatter}
\end{figure}

In Figure \ref{fig:scatter}, we show a scatter plot of coordinate pairs (spread out using a Gaussian filter) for both the Brownian bridge
and the Cauchy bridge. The joint distribution for the Brownian bridge is known. We would like to refer to Felix Xiao's Github page
$$
\text{\url{https://felixxiao.github.io/2018/01/brownian-bridge}}
$$
which is a convenient resource even though all the information there is available in textbooks, such as Karatzas and Shreve.
For the marginal on the argmin, both are uniform by symmetry, as in the cycle method.
It appears that the Cauchy bridge 

In Figure \ref{fig:CBrdgeMinDist} we show the result of the empirical distribution of the minimum of a Cauchy bridge. We have used our hybrid algorithm. We include the Matlab file in Appendix \ref{app:MatlabFiles}.
One notable feature is that the pdf appears to be non-vanishing at $0$. Compare to the cumulative distribution function
for a Brownian bridge, which is $1-e^{-2a^2}$, so that the derivative, pdf, vanishes at $0$.
\begin{figure}
    \centering
    \includegraphics[height=4cm]{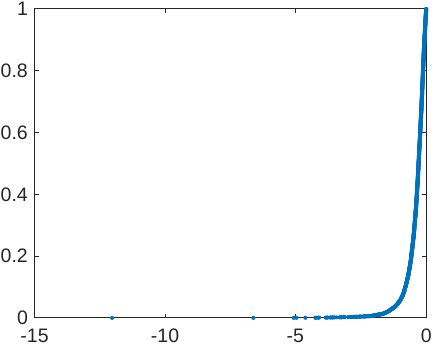}\hspace{1cm}
    \includegraphics[height=4cm]{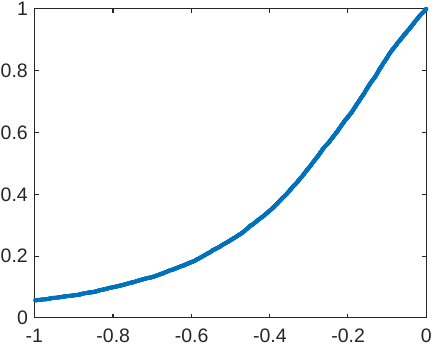}
    \caption{Here we plot the empirical cumulative distribution function for the minimum of the Cauchy bridge, using 10,000 samples, each made using our hybrid algorithm with Ntot=1000 and $\beta=1$. The right picture is a zoom-in of the picture on the left, showing the behavior near 0.}
    \label{fig:CBrdgeMinDist}
\end{figure}
We note that the marginal on the argmin is uniform for the Cauchy bridge, by the same arguments as for the Brownian bridge: the cycle method.
The pictures of the Kolmogorov-Smirnov test are qualitatively similar to KS-test pictures for the Brownian bridge.

\appendix

\section{Limiting $\beta$ cases of the Harmonic MC strategy}

We start by considering the limit $\beta \to 0^+$.
For simplicity of notation, let us denote
\begin{equation}
	\left(\frac{\Delta \gamma}{\Delta t}\right)_{k+\frac{1}{2}}\, =\, \frac{\gamma(t_{k+1})-\gamma(t_k)}
{t_{k+1}-t_k}\, .
\end{equation}
If we replace $\gamma$ by $\beta \gamma$, then it is easy to see that for $k=1,\dots,n-1$
\begin{equation}
	\alpha_{k+1} - 1\, \sim\, \frac{\beta}{\pi}\, \left(\left(\frac{\Delta \gamma}{\Delta t}\right)_{k+\frac{1}{2}}
- \left(\frac{\Delta \gamma}{\Delta t}\right)_{k-\frac{1}{2}}
\right)\, ,
\end{equation}
as $\beta \to 0^+$.
Until otherwise noted, all asymptotic formulas are in the limit $\beta \to 0^+$,
and we will stop writing that.
Then also
\begin{equation}
	\alpha_{1} - \frac{1}{2}\, \sim\, \frac{\beta}{\pi}\, \left(\frac{\Delta \gamma}{\Delta t}\right)_{\frac{1}{2}}\qquad \text{ and }\qquad
	\alpha_{n+1} - \frac{1}{2}\, \sim\, -\frac{\beta}{\pi}\, \left(\frac{\Delta \gamma}{\Delta t}\right)_{n-\frac{1}{2}}\, .
\end{equation}
Then also $\alpha_{n+2}=0$, but that is not needed for the Schwarz-Christoffel formula because
we chose $z_{n+2}=\infty$. Let us choose the base-point for the integral to be $0$. Then
\begin{equation}
	\varphi(z)\, =\, A + C \int_0^z \prod_{k=1}^{n+1} (\zeta - z_k)^{\alpha_k-1}\, d\zeta\, .
\end{equation}
Then we note that, choosing $z_1=0$, $z_{n+1}=1$ and $z_{n+2}=\infty$ (as described before)
we have
\begin{equation}
\prod_{k=1}^{n+1} (\zeta - z_k)^{\alpha_{k}-1}\, =\, 
(\zeta(\zeta-1))^{-1/2} \exp\left(\frac{\beta}{\pi} \sum_{k=0}^{n} \left(
\left(\frac{\Delta \gamma}{\Delta t}\right)_{k+\frac{1}{2}}
- \left(\frac{\Delta \gamma}{\Delta t}\right)_{k-\frac{1}{2}}\right) \log(\zeta-z_{k+1})
\right)\, ,
\end{equation}
where now we define $(\Delta \gamma/\Delta t)_{-1/2}=(\Delta \gamma/\Delta t)_{n+\frac{1}{2}}=0$.
From this we see that we will have $C = -i|C|$ (by consideration of the case $z \to \pm \infty$
along the real axis which is the top of $\mathbb{H}^-$).

With our choice of base-point of the integral, we have $A = 0$. So we may rewrite
\begin{equation}
	\varphi(z)\, =\, - i|C| \int_0^z \prod_{k=1}^{n+1} (\zeta - z_k)^{\alpha_k-1}\, d\zeta\, .
\end{equation}
Since the arcsine law gives the correct formula for $\beta=0$, we must have $|C|-\pi^{-1} = o(1)$
as $\beta \to 0^+$. We will only want to keep the leading order correction term. So we simply replace $|C|$
by $\pi^{-1}$ on the right-hand-side.
For the left-hand-side, let us write
\begin{equation}
	C\, =\, \frac{1}{\pi}\, (1+c)
\end{equation}
where it is understood $c = o(1)$ as $\beta \to 0^+$.
Then Taylor expanding the exponent, using the fact that $\beta$ is small, we have
\begin{equation}
\begin{split}
	\varphi(z) - \frac{2}{\pi}\, \sin^{-1}\left(\sqrt{z}\right)\, &\sim\, \frac{2c}{\pi}\, \sin^{-1}\left(\sqrt{z}\right)\\
&
-\frac{i\beta}{\pi^2}\,
\sum_{k=0}^{n} \left(\left(\frac{\Delta \gamma}{\Delta t}\right)_{k+\frac{1}{2}}
- \left(\frac{\Delta \gamma}{\Delta t}\right)_{k-\frac{1}{2}}\right) 
\int_0^z \frac{\log(\zeta-z_{k+1})}{\sqrt{\zeta(\zeta-1)}}\, d\zeta\, .
\end{split}
\end{equation}
From consideration of $z \to \pm \infty$ along the real axis, we can see that $c$ is real.

We will focus attention on the boundary values, and more precisely $0\leq z\leq 1$.
Therefore, we use the branch of the logarithm already written to change the $1/\sqrt{\zeta(\zeta-1)}$
into $1/(-i\sqrt{\zeta(1-\zeta)})$.
So for $0\leq x\leq 1$, we have
\begin{equation}
\begin{split}
	\varphi(x) - \frac{2}{\pi}\, \sin^{-1}\left(\sqrt{x}\right)\, &\sim\, \frac{2c}{\pi}\, \sin^{-1}\left(\sqrt{x}\right)\\
&
+\frac{\beta}{\pi^2}\,
\sum_{k=0}^{n} \left(\left(\frac{\Delta \gamma}{\Delta t}\right)_{k+\frac{1}{2}}
- \left(\frac{\Delta \gamma}{\Delta t}\right)_{k-\frac{1}{2}}\right) 
\int_0^x \frac{\log(y-z_{k+1})}{\sqrt{y(1-y)}}\, dy\, .
\end{split}
\end{equation}
From this we can see that the leading order value of $z_k$ is
\begin{equation}
z_k^{(0)}\, =\, \sin^2\left(\frac{\pi t_{k-1}}{2}\right)\, .
\end{equation}
This may be used on the right-hand-side of our formula so far, because it is already of order $O(\beta)$
as $\beta \to 0^+$. 
But we are also interested in the leading order behavior of the correction $\xi_k$ in the 
expansion
\begin{equation}
	z_k\, =\, z_k^{(0)} + \xi_k\, ,
\end{equation}
that shows up on the right-hand-side.

So if we take $\ell \in \{2,\dots,n\}$ and replace $x$ by $z_{\ell} = z_{\ell}^{(0)}+\xi_{\ell}$
on the left-hand-side and by $z_{\ell}^{(0)}$ on the right-hand-side, we have
\begin{equation}
\begin{split}
	i\beta \gamma(t_{\ell-1})
-\frac{\xi_{\ell}}{\pi\sqrt{t_{\ell-1}(1-t_{\ell-1})}}\, 
&\sim\, ct_{\ell-1}\\
&\hspace{0.25cm} +\frac{\beta}{\pi^2}\,
\sum_{k=0}^{n} \left(\left(\frac{\Delta \gamma}{\Delta t}\right)_{k+\frac{1}{2}}
- \left(\frac{\Delta \gamma}{\Delta t}\right)_{k-\frac{1}{2}}\right) 
\int_0^{z_{\ell}^{(0)}} \frac{\log(y-z_{k+1}^{(0)})}{\sqrt{y(1-y)}}\, dy\, .
\end{split}
\end{equation}
Let us first consider the imaginary part of both sides. Since $\xi_{\ell}$ must be real
and since $c$ is real, both of those terms are gone. Therefore, the equation will not involve
our free parameters. So we must just have an identity which is true by inspection.
\begin{equation}
\begin{split}
\beta \gamma(t_{\ell-1})\, 
\sim\, \frac{\beta}{\pi^2}\,
\sum_{k=0}^{n} \left(\left(\frac{\Delta \gamma}{\Delta t}\right)_{k+\frac{1}{2}}
- \left(\frac{\Delta \gamma}{\Delta t}\right)_{k-\frac{1}{2}}\right) 
\int_0^{z_{\ell}^{(0)}} \frac{\operatorname{Im}[\log(y-z_{k+1}^{(0)})]}{\sqrt{y(1-y)}}\, dy\, .
\end{split}
\end{equation}
Using our branch of the logarithm (which was imposed to have validity of the SC formula), this gives
\begin{equation}
\begin{split}
\beta \gamma(t_{\ell-1})\, 
\sim\, -\frac{\beta}{\pi}\,
\sum_{k=0}^{n} \left(\left(\frac{\Delta \gamma}{\Delta t}\right)_{k+\frac{1}{2}}
- \left(\frac{\Delta \gamma}{\Delta t}\right)_{k-\frac{1}{2}}\right) 
\int_0^{\min(\{z_{\ell}^{(0)},z_{k+1}^{(0)}\})} \frac{1}{\sqrt{y(1-y)}}\, dy\, .
\end{split}
\end{equation}
Since the integral of $\pi^{-1} (y(1-y))^{-1/2}$ is the arcsine law, and since
the 0th-order pre-vertices were defined to evaluate to the $t$-values under this mapping, 
we see that
\begin{equation}
\begin{split}
\beta \gamma(t_{\ell-1})\, 
\sim\, -\beta\,
\sum_{k=0}^{n} \left(\left(\frac{\Delta \gamma}{\Delta t}\right)_{k+\frac{1}{2}}
- \left(\frac{\Delta \gamma}{\Delta t}\right)_{k-\frac{1}{2}}\right) 
\min(\{t_{\ell-1},t_{k}\})\, .
\end{split}
\end{equation}
But then by summation-by-parts (using $(\Delta \gamma/\Delta t)_r=0$ for $r=-1/2$ and $r=n+\frac{1}{2}$)
\begin{equation}
\begin{split}
\beta \gamma(t_{\ell-1})\, 
\sim\, \beta\,
\sum_{k=0}^{n} \left(\frac{\Delta \gamma}{\Delta t}\right)_{k+\frac{1}{2}}
\left(\min(\{t_{\ell-1},t_{k+1}\})- \min(\{t_{\ell-1},t_{k}\})\right)\, .
\end{split}
\end{equation}
In turn, this can be written as 
\begin{equation}
\begin{split}
\beta \gamma(t_{\ell-1})\, 
\sim\, \beta\,
\sum_{k=0}^{\ell-2} \left(\frac{\Delta \gamma}{\Delta t}\right)_{k+\frac{1}{2}}
\left(t_{k+1}-t_{k}\right)\, =\, \beta \sum_{k=0}^{\ell-2} \left(\gamma(t_{k+1})-\gamma(t_k)\right)\, .
\end{split}
\end{equation}
Since $\gamma(t_{0})=0$, this is a tautologically true equation, using a telescoping sum.
For the imaginary part, the same method would work if we went all the way up to $\ell=n+1$.
So, as required, the imaginary part took care of itself.

Now, for the real part, we have free parameters and we are just trying to determine what those 
parameters are equal to. For $\ell=2,\dots,n$ we have
\begin{equation}
\begin{split}
-\frac{\xi_{\ell}}{\pi\sqrt{t_{\ell-1}(1-t_{\ell-1})}}\, 
&\sim\, ct_{\ell-1}+\frac{\beta}{\pi^2}\,
\sum_{k=0}^{n} \left(\left(\frac{\Delta \gamma}{\Delta t}\right)_{k+\frac{1}{2}}
- \left(\frac{\Delta \gamma}{\Delta t}\right)_{k-\frac{1}{2}}\right) 
\int_0^{z_{\ell}^{(0)}} \frac{\ln(|y-z_{k+1}^{(0)}|)}{\sqrt{y(1-y)}}\, dy\, .
\end{split}
\end{equation}
Let us change variables by taking $y = \sin^2(\pi s/2)$. Then we obtain
\begin{equation}
\begin{split}
-\frac{\xi_{\ell}}{\pi\sqrt{t_{\ell-1}(1-t_{\ell-1})}}\, 
&\sim\, ct_{\ell-1}\\
&\hspace{-1cm}+\frac{\beta}{\pi}\sum_{k=0}^{n} \left(\left(\frac{\Delta \gamma}{\Delta t}\right)_{k+\frac{1}{2}}
- \left(\frac{\Delta \gamma}{\Delta t}\right)_{k-\frac{1}{2}}\right) 
\int_0^{t_{\ell-1}} \ln\left(\left|\sin^2\left(\frac{\pi s}{2}\right)
-\sin^2\left(\frac{\pi t_{k}}{2}\right)\right|\right)\, ds\, .
\end{split}
\end{equation}
Note that if we set $\ell=1$ then the right-hand-side would be $0$, no matter what value we have for $c$.
We set $z_1=0$ and $z_{n+1}=1$, exactly. So we must have $\xi_1=\xi_{n+1}=0$.
So taking $\xi_{n+1}=0$ and substituting in $\ell=n+1$, we obtain
the asymptotic expression for $c$:
\begin{equation}
	c\, \sim\, 
-\frac{\beta}{\pi}\sum_{k=0}^{n} \left(\left(\frac{\Delta \gamma}{\Delta t}\right)_{k+\frac{1}{2}}
- \left(\frac{\Delta \gamma}{\Delta t}\right)_{k-\frac{1}{2}}\right) 
\int_0^{1} \ln\left(\left|\sin^2\left(\frac{\pi s}{2}\right)-\sin^2\left(\frac{\pi t_{k}}{2}\right)\right|\right)\, ds\, .
\end{equation}
Then, using that, for $\ell =2,\dots,n$ we obtain (after more algebraic manipulations)
\begin{equation}
\begin{split}
\xi_{\ell}\, 
&\sim\, \beta (1-t_{\ell-1})^{3/2} t_{\ell-1}^{1/2}
 \sum_{k=0}^{n}\left(\left(\frac{\Delta \gamma}{\Delta t}\right)_{k+\frac{1}{2}}
- \left(\frac{\Delta \gamma}{\Delta t}\right)_{k-\frac{1}{2}}\right)
\int_0^{t_{\ell}} \ln\left(\left|\sin^2\left(\frac{\pi s}{2}\right)-\sin^2\left(\frac{\pi t_{k}}{2}\right)\right|
\right)\, ds\\
&\hspace{-0.5cm} -\beta (1-t_{\ell-1})^{1/2} t_{\ell-1}^{3/2}
\sum_{k=0}^{n}\left(\left(\frac{\Delta \gamma}{\Delta t}\right)_{k+\frac{1}{2}}
- \left(\frac{\Delta \gamma}{\Delta t}\right)_{k-\frac{1}{2}}\right) 
\int_0^{1-t_{\ell}} \ln\left(\left|\sin^2\left(\frac{\pi s}{2}\right)-\sin^2\left(\frac{\pi (1-t_{k})}{2}\right)\right|
\right)\, ds\, .
\end{split}
\end{equation}
In principle, the integrals could be performed in terms of elementary functions and dilogarithms.

\subsection{The large $\beta$ limit: partial results}
For the $\beta \to \infty$ limit, the analysis is more involved.
The measure concentrates on edges incident to minimal vertices.
If there is a unique minimum vertex, located at $a$, then the complex analysis is fairly simple: i.e., if 
\begin{equation}
    \exists k \in \{1,\dots,n-1\}\, ,\ \text{ s.t. }\ W_k<0\ \text{ and }\ \Big( \forall j \in \{1,\dots,n-1\}\setminus \{k\}\ \text{ we have }\ 
        W_j>W_k\Big)\, ,
\end{equation}
and we let $a = t_k$.
Define a function $\varphi : \mathbb{H}^- \to \C$ by the formula
\begin{equation}
    \varphi(a+u)\, =\, a + \frac{1}{\pi}\, \left( (1-a)\, \log\left(\frac{1-a}{1-a-u}\right) - a \log\left(\frac{a+u}{a}\right) \right)\, .
\end{equation}
Then the range is $\mathcal{V} \setminus \Sigma(a)$ where $\mathcal{V}$ is the infinite strip $\{x+iy\, :\, 0<x<1\, ,\ y\in \R\}$ and $\Sigma(a)$ is the ``spike,'' $\{a+iy\, :\, y\geq 0\}$.
\begin{figure}
    \centering
    \begin{tikzpicture}[xscale=5,yscale=1]
        \draw (0.505,-3.25) node[yscale=-1] {\includegraphics[width=5.1cm,height=6.625cm]{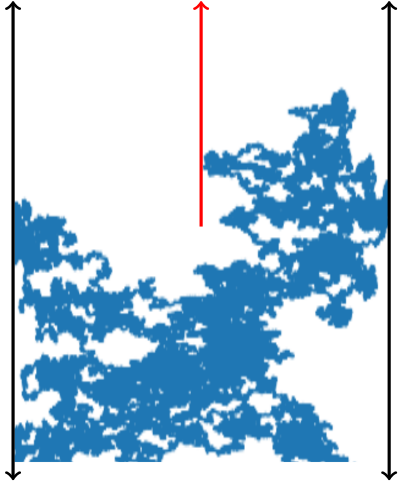}};
        \draw [very thick,red,->] (0.5,-2) -- (0.5,1);
	\draw[very thick,<->] (0,1) -- (0,-5.375);
	\draw[very thick,<->] (1,1) -- (1,-5.375);
    \end{tikzpicture}
    \caption{Here we plot the infinite strip $\mathcal{V} = \{x+iy\, :\, 0<x<1\, ,\ y\in \mathbb{R}\}$ minus the spike $\Sigma(a) = \{a+iy\, :\, y\geq 0\}$. The question is, for a Brownian motion on $\mathcal{V}$, reflecting off the left and right walls, when it hits $\Sigma(a)$
    will it hit from the left or the right side, and at what height $y>0$ will it hit?}
    \label{fig:Spike}
\end{figure}
Then it is easy to see $\varphi'(a)=0$ and $\varphi$ maps $(0,a)$ and $(a,1)$ to the left and right sides of the spike $\Sigma(a)$.
It maps $(-\infty,0)$ and $(1,\infty)$ to the lines $\{0+iy\, :\, y \in \R\}$ and $\{1+iy\, :\, y \in \R\}$.
Note that the branch of the logarithm is such that the logarithm of a unit length complex number (other than $-1$) is between $-i\pi$ and $+i\pi$.

The reason that is relevant is that as $\beta \to \infty$, using the point $a-i\beta W_k$ as the new origin, the local picture will
converge to $\mathcal{V} \setminus \Sigma(a)$. From all of this, we see that 
two edges $(t_{k-1},t_k)$ and $(t_k,t_{k+1})$, incident to $t_k$, are charged with measures
converging to $\frac{2}{\pi} \sin^{-1}(\sqrt{1-a})$ and $\frac{2}{\pi} \sin^{-1}(\sqrt{a})$, respectively.
(We came to this surmise after considering the asymptotics of the Schwarz-Christoffel mapping: there is a way to derive this formula, which
we skip here out of space considerations.)

We have not yet determined the formula for two spikes $\Sigma(a_1)$ and $\Sigma(a_2)$ corresponding to two numbers $k_1<k_2$ with
$k_2-k_1>1$ such that $t_{k_1}$ and $t_{k_2}$ are two points with equal values, equal to the absolute minimum, and such that $W_j>W_{k_1}=W_{k_2}$
for all $j \in \{1,\dots,n-1\} \setminus \{k_1,k_2\}$.
(Note that if $k_2=k_1+1$ then the entire edge would be a set of infinitely local minima, which might be a separate case to consider.)
We hope to do this soon, as well as to determine the analogous formulas
for all higher number of equal absolute minima.

\section{Matlab files for the hybrid models}
\label{app:MatlabFiles}

Here we include the Matlab files of the Hybrid models introduced in the Outlook section.

\subsection{The hybrid model for the Brownian bridge}

{
\footnotesize
\begin{verbatim}
Ntot=100;
beta=0;
xLst = [0,0];
tLst = [0,1];
mdPtLst = (xLst(1)+xLst(2))/2;
DelxLst = xLst(2)-xLst(1);
DelTLst= tLst(2)-tLst(1);
Znew = randn;
zNew = mdPtLst(1)+sqrt(DelTLst/4)*Znew;
xLst = [xLst(1),zNew,xLst(2)];
tLst = [tLst(1),(tLst(1)+tLst(2))/2,tLst(2)];
for nCtr=2:Ntot
    mdPtLst=(xLst(1:nCtr)+xLst(2:(nCtr+1)))/2;
    DelxLst=(xLst(2:(nCtr+1))-xLst(1:nCtr))/2;
    DelTLst=(tLst(2:(nCtr+1))-tLst(1:nCtr));
    medLst = zeros(1,nCtr);
    for kCtr=1:nCtr
        medLst(kCtr)=mdPtLst(kCtr)-sqrt((DelxLst(kCtr)/2)^2+DelTLst(kCtr)*log(2)/2);
    end
    a=min(medLst);
    fLst = zeros(1,nCtr);
    for kCtr=1:nCtr
        fLst(kCtr)=exp(-2*((a-mdPtLst(kCtr))^2-(DelxLst(kCtr)/2)^2)/DelTLst(kCtr));
    end
    pLst = (fLst.^beta.*DelTLst)/sum(fLst.^beta.*DelTLst);
    FLst = cumsum(pLst);
    Urnd = rand;
    Idx = min(find(Urnd<FLst));
    % [M,I]=max(fLst);
    Znew = randn;
    zNew = mdPtLst(Idx)+sqrt(DelTLst(Idx)/4)*Znew;
    xLst = [xLst(1:Idx),zNew,xLst((Idx+1):(nCtr+1))];
    tLst = [tLst(1:Idx),(tLst(Idx)+tLst(Idx+1))/2,tLst((Idx+1):(nCtr+1))];
end


figure
plot(tLst,xLst)
hold on
IntNum = length(tLst)-1;
MuLst = [];
for kCtr = 1:IntNum
    s = tLst(kCtr);
    t = tLst(kCtr+1);
    x = xLst(kCtr);
    y = xLst(kCtr+1);
    alpha = min(x,y);
    beta = abs(x-y);
    tau = abs(t-s);
    U = rand;
    Mu = alpha-0.5*(sqrt(beta^2+2*tau*log(1/(1-U)))-beta);
    MuLst = [MuLst;Mu];
    plot([s,t],[Mu,Mu],'r')
end
\end{verbatim}

}

\subsection{The hybrid model for the Cauchy bridge}

{
\footnotesize
\begin{verbatim}
Ntot=1000;
beta=1;
xLst = [0,0];
tLst = [0,1];
mdPtLst = (xLst(1)+xLst(2))/2;
DelxLst = xLst(2)-xLst(1);
DelTLst= tLst(2)-tLst(1);
tau = DelTLst;
delta = DelxLst; %This equals 0
x = 2*delta/tau; %This equals 0
u = x/2; %This equals 0
Unew = rand;
Gfun = @(r)condlCDF0(r)-Unew;
v = fzero(Gfun,0);
Xnew = mdPtLst+(tau/2)*v;
xLst = [xLst(1),Xnew,xLst(2)];
tLst = [tLst(1),(tLst(1)+tLst(2))/2,tLst(2)];
for nCtr=2:Ntot
    mdPtLst=(xLst(1:nCtr)+xLst(2:(nCtr+1)))/2;
    DelxLst=(xLst(2:(nCtr+1))-xLst(1:nCtr));
    DelTLst=(tLst(2:(nCtr+1))-tLst(1:nCtr));
    medLst = zeros(1,nCtr);
    for kCtr=1:nCtr
        medLst(kCtr)=mdPtLst(kCtr)-0.5*sqrt(2*DelxLst(kCtr)^2+DelTLst(kCtr)^2);
    end
    a=min(medLst);
    fLst = zeros(1,nCtr);
    for kCtr=1:nCtr
        fLst(kCtr)=(DelTLst(kCtr)^2+DelxLst(kCtr)^2)/(DelTLst(kCtr)^2+(2*(mdPtLst(kCtr)-a))^2);
    end
    pLst = (fLst.^beta.*DelTLst)/sum(fLst.^beta.*DelTLst);
    FLst = cumsum(pLst);
    Urnd = rand;
    Idx = min(find(Urnd<FLst));
    % [M,I]=max(fLst);
    tau = DelTLst(Idx);
    delta = DelxLst(Idx);
    x = 2*delta/tau;
    u = x/2;
    Unew = rand;
    Gfun = @(r)condlCDF(u,r)-Unew;
    v = fzero(Gfun,0);
    Xnew = mdPtLst(Idx)+(tau/2)*v;
    xLst = [xLst(1:Idx),Xnew,xLst((Idx+1):(nCtr+1))];
    tLst = [tLst(1:Idx),(tLst(Idx)+tLst(Idx+1))/2,tLst((Idx+1):(nCtr+1))];
end


figure
hold on
IntNum = length(tLst)-1;
MuLst = [];
for kCtr = 1:IntNum
    s = tLst(kCtr);
    t = tLst(kCtr+1);
    x = xLst(kCtr);
    y = xLst(kCtr+1);
    alpha = min(x,y);
    beta = abs(x-y);
    tau = abs(t-s);
    U = rand;
    Mu = alpha+0.5*beta-0.5*sqrt((U*tau^2+beta^2)/(1-U));
    MuLst = [MuLst;Mu];
    plot([s,t],[Mu,Mu],'r.')
end
plot(tLst,xLst,'b.')


function G=condlCDF(u,v)
    G = 0.5 + (1/(2*pi))*(atan(u+v)-atan(u-v)+(1/(2*u))*log((1+(u+v)^2)/(1+(u-v)^2)));
end

function G=condlCDF0(v)
    G = 0.5 + (1/pi)*(atan(v)+v/(1+v^2));
end
\end{verbatim}

}

\section{Algorithms in Sections \ref{sec:2} and \ref{sec:3}} 
Please reference the following link for a GitHub repository containing the Monte Carlo Bisection and Iterative Golden Section Search algorithms, with the associated helper functions.

\begin{center}
\url{https://github.com/erikwu1/OnlineSearchMethods}
\end{center}

\section*{Acknowledgments}
Shannon Starr benefitted from a Simons collaboration grant. He is also grateful for some conversations with Xinghui Xu
which helped inform parts of the Outlook section.

\baselineskip=12pt
\bibliographystyle{plain}

\end{document}